\documentclass[12pt]{amsart}
\makeatletter
\let\@wraptoccontribs\wraptoccontribs
\makeatother

\title{The Donaldson-Thomas partition function of the Banana Manifold}
\author{Jim Bryan}
\contrib[Appendix with ]{Stephen Pietromonaco}

\thanks{This material is based upon work supported by the
National Science Foundation under Grant No. 1440140, while the author
was in residence at the Mathematical Sciences Research Institute in
Berkeley, California, during the Spring semester 2018. The author was
also supported by NSERC.}
\date{\today}

\usepackage[framemethod=tikz]{mdframed} 
\usepackage{tablefootnote} 
\makeatletter 
\AfterEndEnvironment{mdframed}{%
 \tfn@tablefootnoteprintout%
 \gdef\tfn@fnt{0}%
}
\makeatother

\usepackage{extsizes}
\usepackage{blindtext}

\usepackage{amsmath}
\usepackage{amsmath,amsthm,amsfonts,amssymb}
\usepackage{times}
\usepackage{relsize}

\usepackage{tikz-cd}
\usepackage{pgfplots}
\usepackage{tikz}
\usepackage{verbatim}

\usetikzlibrary{calc,3d}

\definecolor{tealgreen}{HTML}{1B9E77}
\definecolor{orange}{HTML}{D95F02}
\definecolor{purple}{HTML}{7570B3}
\definecolor{pink}{HTML}{E7298A}
\definecolor{grassgreen}{HTML}{66A61E}
\definecolor{goldyellow}{HTML}{E6AB02}
\definecolor{brown}{HTML}{A6761D}
\definecolor{devilgray}{HTML}{666666}

\newtheorem{theorem}{Theorem}
\newtheorem{proposition}[theorem]{Proposition}

\newtheorem{lemma}[theorem]{Lemma}
\newtheorem{corollary}[theorem]{Corollary}

\newtheorem{def-theorem}[theorem]{Definition-Theorem}

\theoremstyle{definition}

\newtheorem{remark}[theorem]{Remark}
\newtheorem{definition}[theorem]{Definition}

\newcommand{\CC} {{\mathbb C}}          
\newcommand{\NN} {{\mathbb N}}		
\newcommand{\ZZ} {{\mathbb Z}}		
\newcommand{\PP} {{\mathbb P}}		
\newcommand{\HH} {{\mathbb H}}		

\newcommand{\Li}{\operatorname{Li}}
\newcommand{\Ell}{\operatorname{Ell}}

\newcommand{\Ker}{\operatorname{Ker}}

\newcommand{\Hilb}{\operatorname{Hilb}}
\newcommand{\Sym}{\operatorname{Sym}}
\newcommand{\Spec}{\operatorname{Spec}}
\newcommand{\Tot}{\operatorname{Tot}}

\newcommand{\Mod}{\mathsf{Mod}}
\newcommand{\Jac}{\mathsf{Jac}}
\newcommand{\Siegel}{\mathsf{Siegel}}

\newcommand{\norm}{\mathsf{norm}}
\newcommand{\sing}{\mathsf{sing}}
\newcommand{\Sch}{\mathsf{Sch}}
\newcommand{\ban}{\mathsf{ban}}
\newcommand{\inst}{\mathsf{prod}}
\newcommand{\hook}{\mathsf{hook}}
\newcommand{\DT}{\mathsf{DT}}
\newcommand{\DTss}{ \, }

\newcommand{\VertexTilde}{\widetilde{\mathsf{V}}}
\newcommand{\OO}{\mathcal{O}}

\newcommand{\Bl}{\operatorname{Bl}}
\newcommand{\dvec}{\mathbf{\underline{d}}}
\newcommand{\avec}{\mathbf{\underline{a}}}
\newcommand{\betadvec}{\beta_{\dvec}}
\newcommand{\zerovec}{\mathbf{\underline{0}}}
\newcommand{\half}{\frac{1}{2}}
\newcommand{\Var}{\operatorname{Var}}
\newcommand{\Vsf}{\mathsf{V}}
\newcommand{\Vtildesf}{\widetilde{\Vsf}}
\newcommand{\ptotheminusrho}{p^{-\rho}}
\newcommand{\betacomman}{\beta, n }

\newcommand{\Fhat}{\widehat{F}}
\newcommand{\Ohat}{\widehat{\mathcal{O}}}
\newcommand{\HilbHat}{\widehat{\Hilb}}
\newcommand{\MW}{\mathsf{MW}}
\newcommand{\Conf}{\operatorname{Conf}}
\newcommand{\Chow}{\operatorname{Chow}}
\newcommand{\Pic}{\operatorname{Pic}}
\newcommand{\Stab}{\operatorname{Stab}}
\newcommand{\Ext}{\operatorname{Ext}}

\begin{document}

\begin{abstract}
A banana manifold is a compact Calabi-Yau threefold, fibered by
Abelian surfaces, whose singular fibers have a singular locus given by
a ``banana configuration of curves''. A basic example is given by
$X_{\ban}:=\Bl_{\Delta}(S\times_{\PP^{1}}S)$, the blowup along the diagonal of
the fibered product of a generic rational elliptic surface $S\to
\PP^{1}$ with itself.

In this paper we give a closed formula for the Donaldson-Thomas
partition function of the banana manifold $X_{\ban }$ restricted to
the 3-dimensional lattice $\Gamma$ of curve classes supported in the
fibers of $X_{\ban}\to \PP^{1}$. It is given by
\[
Z^{\DTss}_{\Gamma}(X_{\ban}) = \prod_{d_{1},d_{2},d_{3}\geq 0} \prod_{k}
\left(1-p^{k}Q_{1}^{d_{1}}Q_{2}^{d_{2}}Q_{3}^{d_{3}}\right)^{-12c(||\dvec ||,k)}
\]
where $||\dvec || = 2d_{1}d_{2}+ 2d_{2}d_{3}+
2d_{3}d_{1}-d_{1}^{2}-d_{2}^{2}-d_{3}^{2}$, and the coefficients
$c(a,k)$ have a generating function given by an explicit ratio of
theta functions.  This formula has interesting properties and is
closely related to the equivariant elliptic genera of $\Hilb
(\CC^{2})$. In an appendix with S. Pietromonaco, it is shown that the
corresponding genus $g$ Gromov-Witten potential $F_{g}$ is a
genus 2 Siegel modular form of weight $2g-2$ for $g\geq 2$, namely it
is the Skoruppa-Maass lift of a multiple of an Eisenstein series,
namely $\frac{6|B_{2g}|}{g(2g-2)!} E_{2g}(\tau )$.
\end{abstract}

\maketitle 


\tableofcontents

\section{Introduction}

\subsection{Donaldson-Thomas invariants.}

The Donaldson-Thomas invariants of a Calabi-Yau threefold $X$ encode
subtle information about the enumerative geometry of $X$. They are a
mathematical incarnation of counts of BPS states in B-model
topological string theory compactified on $X$.

Let $X$ be a Calabi-Yau threefold, that is a smooth complex threefold
with trivial canonical class. Let $\beta \in H_{2}(X,\ZZ )$ be a curve
class, let $n\in \ZZ$, and let 
\[
\Hilb^{\beta ,n}(X) = \left\{Z\subset X:\quad [Z] = \beta ,\, \chi (\OO_{Z})=n \right\}
\]
be the Hilbert scheme parameterizing dimension one subschemes in the
class $\beta$ and having holomorphic Euler characteristic $n$. The
Donaldson-Thomas invariant $\DT_{\beta ,n}(X)$ can be defined
\cite{Behrend-Micro} as a weighted Euler characteristic of the Hilbert scheme:
\[
\DT_{\beta ,n}(X) = e\left(\Hilb^{\beta ,n}(X),\nu \right):= \sum_{k\in
\ZZ} k\cdot e\left(\nu^{-1}(k) \right)
\]
where $e(\cdot)$ is topological Euler characteristic and 
\[
\nu :\Hilb^{\beta ,n}(X) \to \ZZ 
\]
is Behrend's constructible function. One can regard $\DT_{\beta ,n}(X)$
as a virtual count of the number of curves in the class $\beta$ with Euler
characteristic $n$.

The Donaldson-Thomas partition function is a generating function for
the invariants which we define\footnote{Our insertion of a minus sign
is somewhat non-standard.} as
\[
Z^{\DTss}(X) = \sum_{\beta \in H_{2}(X,\ZZ )}\sum_{n\in \ZZ} \DT_{\beta
,n}(X)\, Q^{\beta} (-p)^{n}.
\]

Here $Q^{\beta} = Q_{1}^{d_{1}}\dotsb  Q_{r}^{d_{r}}$ where $\beta
=d_{1}C_{1}+\cdots +d_{r}C_{r}$ and $\{C_{1},\dots ,C_{r} \}$ is a
basis for $H_{2}(X,\ZZ )$ chosen so that $d_{i}\geq 0$ for all
effective curve classes. $Z^{\DTss}(X)$ is then a formal power series in
$Q_{1},\dots ,Q_{r}$ whose coefficients are formal Laurent series in
$p$.

The Donaldson-Thomas partition function is very hard to compute:

\vspace{.25in}

\begin{mdframed}[backgroundcolor=lightgray]
Currently, there is not a single compact Calabi-Yau threefold $X$
for which $Z^{\DTss}(X)$ is completely known, not even
conjecturally\tablefootnote{Strictly speaking, we do know (for elementary
reasons) that $Z^{\DTss}(X)=1$ when $X$ is an Abelian threefold, a
product of a $K3$ surface and an elliptic curve, or any other
threefold with a free action by an Abelian variety. There is a
modification of the Donaldson-Thomas invariants which is non-trivial
in those cases, and for Abelian threefolds, we do have a complete, but
conjectural, answer for the partiton function of the modified
invariants \cite{BOPY}.}.
\end{mdframed}

\vspace{.25in}

Given some sublattice $\Gamma \subset H_{2}(X,\ZZ )$, we can define a
restricted partition function:
\[
Z^{\DTss}_{\Gamma}(X)  =  \sum_{\beta \in \Gamma }\sum_{n\in \ZZ} \DT_{\beta
,n}(X)\, Q^{\beta} (-p)^{n}.
\]

Even for restricted partition functions, there are very few results
for compact $X$. For elliptic or $K3$ fibrations $\pi :X \to B$, the
restricted partition functions $Z^{\DTss}_{\Gamma}(X)$ have been computed
for $\Gamma =\Ker (\pi_{*}:H_{2}(X)\to H_{2}(B))$, i.e. fiber classes,
by Toda \cite[Thm 6.9]{Toda-2012-Kyoto} in the case of elliptic
fibrations and Maulik, Pandharipande, and Thomas
\cite{Maulik-Pandharipande-Thomas,Pandharipande-Thomas-KKV} in the
case of $K3$ fibrations. In the case of $K3$ fibrations, the partition
functions exhibit modularity properties.

In this paper we compute $Z_{\Gamma}(X)$ where $X$ is a certain kind
of Calabi-Yau threefold (a banana manifold) and $\Gamma$ is a rank 3
lattice. We give an explicit product formula for $Z_{\Gamma}(X)$, we
derive a generating function for the corresponding Gopakumar-Vafa
invariants, and we show (assuming the GW/DT correspondence) that the
associated Gromov-Witten potentials are Siegel modular forms.

\subsection{Banana manifolds and their partition functions}

We study a class of compact Calabi-Yau threefolds which we call banana
manifolds. The basic example, which we denote $X_{\ban}$, is defined
as follows\footnote{ Essentially this construction of the banana
manifold is mentioned briefly in \cite[End of section
5.2]{Kanazawa-Lau}; we learned about this from Georg Oberdieck.  }.
\begin{definition}\label{def: banana manifold}
Let $S\to \PP^{1}$ be a generic rational elliptic surface. Let
\[
X_{\ban} = \Bl_{\Delta}\left(S\times_{\PP^{1}}S \right)
\]
be the fibered product of $S$ with itself blown up along the diagonal
$\Delta$. See Figure~\ref{fig: the banana manifold}. 
\end{definition}

\begin{center}

\begin{figure}
\centering
\begin{tikzpicture}[
                    z  = {-15},
		    scale = 0.75]

\begin{scope}[yslant=-0.3,xslant=0]


\begin{scope} [canvas is yz plane at x=0]
\draw [black](0,0) rectangle (3,5);
\draw [pink, ultra thick, domain=0:3] 
(0,0)--(3,5);
\end{scope}
\begin{scope} [canvas is xz plane at y=0]
\draw [black](0,0) rectangle (4,5);
\end{scope}
\draw [black](0,0) rectangle (4,3);

\draw [ultra thick,orange] 
                   (1  ,0   ,5)
to [out=90,in=-90] (1  ,0.6 ,5)
to [out=90,in=-90] (0.5,1.5 ,5)
to [out=90,in=-90] (1  ,2.4 ,5)
to [out=90,in=-90] (1  ,3   ,5);
\draw [ultra thick,orange] 
                   (2  ,0   ,5)
to [out=90,in=-90] (2  ,0.6 ,5)
to [out=90,in=-90] (1.5,1.5 ,5)
to [out=90,in=-90] (2  ,2.4 ,5)
to [out=90,in=-90] (2  ,3   ,5);


\draw [ultra thick,pink] (0,3,5)--(4,3,5);
\draw [ultra thick,pink] (0,0,0)--(4,0,0);

\draw [ultra thick,orange] 
                    (3   ,0   ,5) 
to [out=90,in=0]    (2.3 ,1.8 ,5) 
to [out=180,in=90]  (2   ,1.5 ,5) 
to [out=270,in=180] (2.3 ,1.2 ,5) 
to [out=0,in=270]   (3   ,3   ,5);

\node [right] at (2.8,1.5,5) {$n_{j}$};
\node [left] at (0,2.0,5) {$S$};
\node [left] at (0,0.5,5) {$B$};


\begin{scope} [canvas is yz plane at x=4]
\draw [black](0,0) rectangle (3,5);
\draw [pink, ultra thick, domain=0:3] 
(0,0)--(3,5);
\draw(1.5,2.5) node[above]{$\Delta $};
\end{scope}

\begin{scope} [canvas is xz plane at y=3]
\draw [black](0,0) rectangle (4,5);
\draw [ultra thick,orange]
(3,0)
to (3,1) [out=90,in=0]
to (2.3,3.0) [out=180,in=90]
to (2.0,2.5) [out=-90,in=180]
to (2.3,2.0) [out=0,in=-90]
to (3,4) [out=90,in=-90]
to (3,5);
\draw [ultra thick,orange]
(1,0)
to (1,1) [out=90,in=-90]
to (0.5,2.5) [out=90,in=-90]
to (1,4) [out=90,in=-90]
to (1,5);
\draw [ultra thick,orange]
(2,0)
to (2,1) [out=90,in=-90]
to (1.5,2.5) [out=90,in=-90]
to (2,4) [out=90,in=-90]
to (2,5);

\end{scope}

\draw [black](0,0,5) rectangle (4,3,5);
\draw [black,fill, opacity=0.1](0,0,5) rectangle (4,3,5);

\end{scope}
\end{tikzpicture}
\caption{The banana manifold $X_{\ban}=\Bl(S\times_{\PP^{1}}S)$. 
}\label{fig: the banana manifold}
\end{figure}
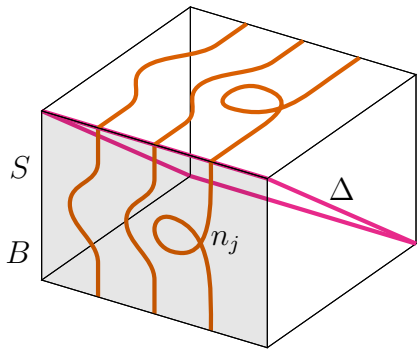

\end{center}

The map $S\to \PP^{1}$ is singular at 12 points which gives rise to 12
conifold singularities in the fibered product which all lie on the
divisor $\Delta$. Consequently, $X_{\ban}$ is a conifold resolution of
$S\times_{\PP^{1}}S$. $X_{\ban}$ is a simply connected, compact
Calabi-Yau threefold with Hodge numbers $h^{1,1}(X_{\ban})=20$ and
$h^{1,2}(X_{\ban})=8$ (see \S~\ref{sec: geom of ban man}).

The generic fiber of $\pi :X_{\ban} \to \PP^{1}$ is $E\times E$, the
product of an elliptic curve with itself and $\pi$ has 12 singular
fibers which are non-normal toric surfaces, $F_{\sing}$, each of which
is a compactification of $\CC^{*}\times \CC^{*}$ by a \emph{banana
configuration} of $\PP^{1}$'s.
\begin{definition}\label{defn: banana configuration}
A \emph{banana configuration}\footnote{The SYZ mirror of this
configuration was studied by Abouzaid, Auroux, and Katzarkov in
\cite{Abouzaid-Auroux-Katzarkov}. The banana configuration also
appears in the Landau-Ginzburg mirror of a genus two curve as studied
by Gross, Katzarkov, and Ruddat in \cite{Gross-Katzarkov-Ruddat} and
Ruddat in \cite{Ruddat}. } in a Calabi-Yau threefold $X$ is a union
$C=C_{1}\cup C_{2}\cup C_{3}$ of three curves $C_{i}\cong \PP^{1}$
with $N_{C_{i}/X}\cong \OO (-1)\oplus \OO (-1)$ and such that
$C_{1}\cap C_{2}=C_{2}\cap C_{3}=C_{3}\cap C_{1}=\{p,q \}$ where
$p,q\in X$ are distinct points. Moreover, there exist coordinates on
formal neighborhoods of $p$ and $q$ such that the curves $C_{i}$ are
given by the coordinate axes in those coordinates. See
figure~\ref{fig: banana configuration}; the meaning of the picture on
the right is discussed in section~\ref{subsec: Mordell weil gps and actions}.
\end{definition}

\begin{figure}
\centering
\begin{tikzpicture}[xshift=5cm,
		    scale = 1.0
		    ]

\begin{scope}  
\draw (0,0) ellipse (2.4 and 2);
\draw (0,0) ellipse (0.6cm and 2cm);
\draw (0,0) ellipse (1.2cm and 2cm);
\draw (-0.6,0) arc(180:360:0.6 and 0.3);
\draw[dashed](-0.6,0) arc(180:0:0.6cm and 0.3cm);
\draw (1.2,0) arc(180:360:0.6cm and 0.3cm);
\draw[dashed](1.2,0) arc(180:0:0.6cm and 0.3cm);
\draw (-2.4,0) arc(180:360:0.6cm and 0.3cm);
\draw[dashed](-2.4,0) arc(180:0:0.6cm and 0.3cm);

\draw (0,-2) node[below] {$p$};
\draw(0,2) node[above] {$q$};
\draw(0,0.6) node {$C_{2}$};
\draw(-1.8,0.6) node {$C_{1}$};
\draw(1.8,0.6)node {$C_{3}$};
\end{scope}


\begin{scope}[xshift=4.5cm,yshift=-2cm]
\draw (0,1.5)--(1.5,1.5)--(2.5,2.5)--(4,2.5);
\draw (1.5,0)--(1.5,1.5)--(2.5,2.5)--(2.5,4);
\draw (0.5,1.5)node{$||$};
\draw (3.5,2.5)node{$||$};
\draw (1.5,0.5)node{$-$};
\draw (2.5,3.5)node{$-$};
\draw (2.5,3.1)node[left]{$C_{3}$};
\draw (2,2)node[below right]{$C_{2}$};
\draw (0.9,1.5)node[above]{$C_{1}$};
\draw (1.5,1.5)node[below left]{$p$};
\draw (2.5,2.5)node[above right]{$q$};
\end{scope}

\end{tikzpicture}
\caption{The banana configuration of $\PP^{1}$'s.  
}
\label{fig: banana configuration}
\end{figure}
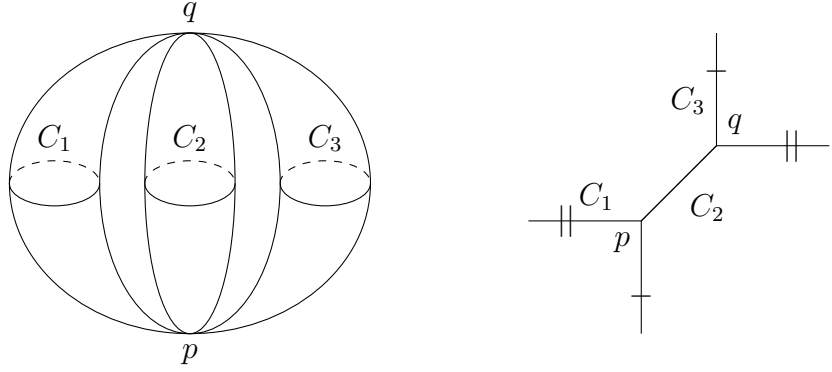

The banana curves $C_{1}, C_{2}, C_{3}$ of any of the singular fibers
generate 
\[
\Gamma= \Ker \pi_{*}\subset H_{2}(X_{\ban},\ZZ ),
\]
the lattice of the fiber classes (Lemma~\ref{lem: C1,C2,C3 span fiber classes}). Let
\[
\beta_{\dvec} = d_{1}C_{1}+ d_{2}C_{2}+ d_{3}C_{3}
\]
and define
\[
||\dvec || = 2d_{1}d_{2}+ 2d_{2}d_{3}+ 2d_{3}d_{1}-d_{1}^{2}-d_{2}^{2}-d_{3}^{2}.
\]
This quadratic form is twice the intersection form of a smoth fiber,
by which we mean that if $\beta_{\dvec}$ is represented by a cycle $C$
supported on a smooth fiber, then $||\dvec || = 2 C\cdot C$ where
$C\cdot C$ is the self-intersection of $C$ in the smooth fiber.

We define banana manifolds in general as follows.
\begin{definition}\label{defn: banana manifolds}
We say that a compact Calabi-Yau threefold $X$ is a \emph{banana
manifold} with $N$ banana fibers if there is an Abelian surface
fibration $\pi :X\to \PP^{1}$ where the singular locus of $\pi$
consists of $N$ disjoint banana configurations and the smooth locus of
$\pi$ is an Abelian group scheme over $\PP^{1}$ whose natural action
on itself extends to an action on $X$ (see \cite{Bryan-Pietromonaco}
for examples of rigid Banana manifolds). 
\end{definition}

Our main theorem is the following:
\begin{theorem}\label{thm: formula for Zban}
Let $X_{\ban }  $ be the basic banana
manifold and let $\Gamma  = \Ker \pi_{*}\cong \ZZ^{3}$ where $\pi
:X_{\ban}\to \PP^{1}$. Then the Donaldson-Thomas partition function of
$X_{\ban}$, restricted to $\Gamma$ is given by
\[
Z^{\DTss}_{\Gamma}(X_{\ban }) = \prod_{d_{1},d_{2},d_{3}\geq 0} \prod_{k}
\left(1-p^{k}Q_{1}^{d_{1}}Q_{2}^{d_{2}}Q_{3}^{d_{3}}\right)^{-12c(||\dvec ||,k)}
\]
where the second product is over all $k\in \ZZ$ unless
$(d_{1},d_{2},d_{3})=(0,0,0)$ in which case $k>0$, and where the
$c(||\dvec ||,k)$ are positive integers given by
\[
\sum_{a=-1}^{\infty} \,\,\sum_{k\in \ZZ} c(a,k) Q^{a}
y^{k} = \frac{\sum_{k\in \ZZ}
Q^{k^{2}}(-y)^{k}}{\left(\sum_{k\in \ZZ +\half} Q^{2k^{2}}(-y)^{k}
\right)^{2}}  = \frac{\vartheta_{4}(2\tau ,z)}{\vartheta_{1}(4\tau ,z)^{2}}
\]
where $\vartheta_{4}$ and $\vartheta_{1}$ are the usual theta
functions\footnote{Following the conventions from Wolfram Alpha's
Jacobi Theta Function page.} with
$Q = e^{2\pi i \tau }$ and $y=e^{2 \pi i z
}$.
\end{theorem}

\begin{remark}\label{rem: main thm works for general bananafolds}
While we have formulated the Theorem for the basic
banana manifold, our proof will work for banana manifolds in general,
where the 12 in the exponent on the right hand side will be replaced
by $N$, the number of banana configurations. 
\end{remark}

\begin{remark}
The coefficients $c(a,k)$ are also the coefficients of the equivariant
elliptic genus of the plane $\CC ^{2}$, see \S~\ref{sec: elliptic
genera of Hilb(C2)}.  We also note that the right hand side of the
above equation, is a meromorphic Jacobi form of weight $-\half$ and
index -1 for the group $\Gamma (4)$.
\end{remark}

\begin{corollary}[Pietromonaco]\label{cor: Fg's are Siegel forms (Stephen)}
Assuming the Gromov-Witten/Donaldson-Thomas correspondence holds for a
banana manifold $X$, the genus $g$ Gromov-Witten potential function
$F_{g}(Q_{1},Q_{2},Q_{3})$ is a meromorphic Siegel modular form of weight $2g-2$
for all $g\geq 2$ where $Q_{1}=e^{2\pi i z }$, $Q_{2}=e^{2\pi i (\tau
-z)} $, $Q_{3}=e^{2\pi i (\sigma -z)}$, and where
$\left(\begin{smallmatrix} \tau &z\\z&\sigma \end{smallmatrix}
\right)\in \mathbb{H}_{2}$ is in the genus 2 Siegel upper half
plane. Namely, $F_{g}$ is given by the Skoruppa-Maass lift
of $a_{g}E_{2g}(\tau )$, the $2g$-th Eisenstein series multiplied by
the constant $a_{g} = \frac{6|B_{2g}|}{g(2g-2)!}$. See
Appendix~\ref{Appendix: GW potentials are Siegel modular forms} for
full definitions of the terms.
\end{corollary}

This corollary has a natural interpretation in terms of mirror
symmetry, see Remark~\ref{rem: Fg are genus 2 Siegel forms have mirror
symmetry interpretation}.   

Our Theorem~\ref{thm: formula for Zban} also completely determines the
Gopakumar-Vafa invariants of the banana manifold. One corollary is:

\begin{corollary}\label{cor: GV(Xban) only depend on ||beta||}
The Gopakumar-Vafa invariants of $X_{\ban}$ in the class
$\beta_{\dvec}\, $ only depend on $||\dvec ||$. We thus streamline the
notation by writing
\[
n^{g}_{\beta_{\dvec}}(X_{\ban}) = n^{g}_{a}(X_{\ban})\quad \text{where
$a=||\dvec ||$.}
\]
\end{corollary}

The property that $n^{g}_{\beta_{\dvec}}(X_{\ban})$ only depends on
$||\dvec ||$ and so in particular the invariants are independent of
the divisibility of $\beta_{\dvec}$, is an unusual property of
$X_{\ban}$ which is also shared by the local $K3$ surface by a deep
result of Pandharipande-Thomas \cite{Pandharipande-Thomas-KKV}.

We can reformulate our main result in terms of the Gopakumar-Vafa
invariants. After some manipulation of generating functions (see
\S~\ref{sec: elliptic genera of Hilb(C2)} and \S~\ref{sec: BPS
invariants from product DT partition function}) we can deduce:

\begin{theorem}\label{thm: generating function for MT polys of Xban}
The Gopakumar-Vafa invariants $n^{g}_{\beta_{\dvec}}(X_{\ban}) =
n^{g}_{a}(X_{\ban})$ in the classes $\beta_{\dvec} \, $ with $||\dvec
||=
a$ are given by
\[
\sum_{a=-1}^{\infty} \sum_{g\geq 0} n_{a}^{g}(X_{\ban})
\left(y^{\half}+y^{-\half} \right)^{2g} Q^{a+1} =
12\prod_{n=1}^{\infty}
\frac{(1+yQ^{2n-1})(1+y^{-1}Q^{2n-1})(1-Q^{2n})}
{(1+yQ^{4n})^{2}(1+y^{-1}Q^{4n})^{2}(1-Q^{4n})^{2}}.
\]
\end{theorem}

It is interesting to compare the above formula to the analogous result for
the local $K3$ surface, namely the Katz-Klemm-Vafa formula:

\begin{theorem}[Pandharipande-Thomas\cite{Pandharipande-Thomas-KKV}]\label{thm: KKV formula}

The Gopakumar-Vafa invariants $n^{g}_{\beta}(K3) =
n^{g}_{a}(K3)$ in the class $\beta$ with $\beta^{2}/2 = a$ are given by
\[
\sum_{a=-1}^{\infty} \sum_{g\geq 0} n_{a}^{g}(K3)
\left(y^{\half}+y^{-\half} \right)^{2g} Q^{a+1} = \prod_{n=1}^{\infty}
\frac{1}{(1+yQ^{n})^{2}(1+y^{-1}Q^{n})^{2}(1-Q^{n})^{20}}.
\]
\end{theorem}

Note that the right hand side of the above equation, after the
substitution $Q=-yq$, is equal to $\sum_{n}\chi_{y}(\Hilb^{n}
(K3))q^{n}$, the generating function for the $\chi_{y}$-genus of the
Hilbert schemes of points on $K3$. It would be interesting to find an
analogous interpretation of the right hand side of the equation in
Theorem~\ref{thm: generating function for MT polys of Xban}.

See \S~\ref{Appendix: GW potentials are Siegel modular forms} for a
table of $n_{a}^{g}(X_{\ban})$ for small values of $a$ and $g$.

\section{The coefficients $c(a,k)$ and the elliptic genera
of $\Hilb(\CC^{2})$. }\label{sec: elliptic genera of Hilb(C2)}

Let $M$ be a compact complex manifold of dimension $d$ and let $x_{1},\dotsc
,x_{d}$ be the Chern roots of $TM$. Then the elliptic genus is defined
by
\[
\Ell_{q,y} (M) = \int_{M}\,\, \prod_{j=1}^{d} x_{j} y^{-\half }\,\,
\prod_{n=1}^{\infty} \frac{\left(1-ye^{-x_{j}}q^{n-1}
\right)\left(1-y^{-1}e^{x_{j}}q^{n} \right)}{\left( 1-e^{-x_{j}}q^{n-1} \right)\left( 1-e^{x_{j}}q^{n} \right)} .
\]

If $M$ has a $\CC^{*}$ action with isolated fixed points, then by
Atiyah-Bott localization we get
\begin{equation}\label{eqn: formula for Ell via Atiyah-Bott localization}
\Ell_{q,y}(M) = \sum_{p\in M^{\CC^{*}}} \,\,\prod_{j=1}^{d} y^{-\half}
\prod_{n=1}^{\infty} \frac{\left(1-y t^{-k_{j}(p)}q^{n-1} \right)\left(1-y^{-1} t^{k_{j}(p)}q^{n} \right)}{\left(1- t^{-k_{j}(p)}q^{n-1} \right)\left(1- t^{k_{j}(p)}q^{n} \right)}
\end{equation}
where $k_{1}(p),\dotsc ,k_{d}(p)\in \ZZ$ are the weights of the
$\CC^{*}$ action on $T_{p}M$.

If $M$ is non-compact, we take equation \eqref{eqn: formula for Ell
via Atiyah-Bott localization} to be the definition of the elliptic
genus $\Ell_{q,y}(M,t)$ which then may depend on $t$, the equivariant
parameter\footnote{The parameter $t\in H^{*}_{\CC^{*}}(pt)$ is the
Chern character of the universal line bundle so $t=e^{c_{1}}$ where
$c_{1}\in H^{2}_{\CC^{*}}(pt)$ is an integral generator. }.

It is convenient to rewrite this expression in terms of the Fourier
expansion of the theta
function $\theta_{1}$. Let $q=\exp\left(2\pi i \tau \right)$ and
$y=\exp\left(2\pi i z \right)$, then $\theta_{1}(q,y)$ is given by
\begin{align*}
\theta_{1}(q,y) = &-\sum_{k\in \ZZ +\half} q^{\frac{k^{2}}{2}} (-y)^{k}\\
=& -iq^{\frac{1}{8}}y^{-\frac{1}{2}} \prod_{n=1}^{\infty}
(1-q^{n})(1-yq^{n-1})(1-y^{-1}q^{n}) .
\end{align*}
Then equation~\eqref{eqn: formula for Ell via Atiyah-Bott
localization} becomes
\[
\Ell_{q,y}(M,t) = \sum_{p\in M^{\CC^{*}}} \prod_{j=1}^{d} \frac{\theta_{1}(q,yt^{-k_{j}(p)})}{\theta_{1}(q,t^{-k_{j}(p)})}
\]

For example, if we let $\CC^{*}$ act on $\CC^{2}$ with weights $\pm
1$, then the induced action of $\CC^{*}$ on $\Hilb^{m}(\CC^{2})$ has
isolated fixed points corresponding to monomial ideals which are in
bijective correspondence with integer partitions of $m$. The tangent
weights associated to a partition $R$ are given by $\{\pm h_{j,k}\}$
where $h_{j,k}=h_{j,k}(R)$ is the hook length of the box in position
$(j,k)$ in the diagram of $R$ (see \S~\ref{subsec: Notation and schur
function identities} for this notation). Thus

\begin{equation}\label{eqn: Ell(Hill(C2)) as a sum over partitions}
\sum_{m=0}^{\infty} \Ell_{q,y}(\Hilb^{m}(\CC^{2}),t)Q^{m} = \sum_{R}
Q^{|R|}  \prod_{(j,k)\in R}
 \frac{\theta_{1}(q,yt^{-h_{j,k}})\theta_{1}(q,yt^{h_{j,k}})} {\theta_{1}(q,t^{-h_{j,k}})\theta_{1}(q,t^{h_{j,k}})}.
\end{equation}

Dijkgraaf-Moore-Verlinde-Verlinde conjectured \cite{DMVV} that 
\[
\sum_{m=0}^{\infty} \Ell_{q,y}(\Hilb^{m}(\CC^{2}),t)\,  Q^{m} =
\sum_{m=0}^{\infty} \Ell^{orb}_{q,y}(\Sym^{m}(\CC^{2}),t)\, Q^{m}
\]
where $ \Ell^{orb}_{q,y}(\Sym^{m}(\CC^{2}),t)$ is the orbifold
elliptic genera of $\Sym^{m}(\CC^{2})$. This is a special case of the
crepant resolution conjecture for equivariant elliptic genera which
was proven by Waelder \cite[Thm~12]{Waelder} based on the
non-equivariant case proven by Borisov-Libgobner
\cite{Borisov-Libgober}. The right hand side of the DMVV conjecture is
easy to compute; consequently, Waelder's result leads to the following
formula:

\begin{theorem}[Waelder]\label{thm: DMVV formula}
\[
\sum_{m=0}^{\infty} \Ell_{q,y}(\Hilb^{m}(\CC^{2}),t)Q^{m} =
\prod_{m=1}^{\infty }\prod_{n=0}^{\infty}\prod_{l,k\in \ZZ} \left(1 -
t^{k}q^{n}y^{l}Q^{m} \right)^{-c(nm,l,k)}
\]
where the coefficients $c(n,l,k)$ are defined by
\begin{align*}
\Ell_{q,y} (\CC^{2},t) &= \frac{\theta_{1}(q,yt)\theta_{1}(q,yt^{-1})}{\theta_{1}(q,t)\theta_{1}(q,t^{-1})} \\
&= \sum_{n=0}^{\infty } \sum_{k,l\in \ZZ} c(n,l,k)q^{n}y^{l}t^{k}.
\end{align*}
\end{theorem}

We prove the following result about the coefficients $c(n,l,k)$.
\begin{proposition}\label{prop: generating function for c(a,k) =
c(4n-l^2,k) =c(n,l,k)}
Let 
\[
\Ell_{q,y}(\CC^{2},t)
=\sum_{n=0}^{\infty}\sum_{l,k\in \ZZ} c(n,l,k)q^{n} y^{l}t^{k} .
\]
Then $c(n,l,k)$ only depends on the pair $(4n-l^{2},k)$. Writing
\[
c(n,l,k) = c(4n-l^{2},k)
\]
we have $c(a,k)=0$ if $a<-1$ and 
\[
\sum_{a=-1}^{\infty } \sum_{ k\in \ZZ} c(a,k)Q^{a}t^{k} = \frac{\sum_{k\in
\ZZ}Q^{k^{2}}(-t)^{k}}{\left(\sum_{k\in \ZZ +\half } Q^{2k^{2}}
(-t)^{k} \right)^{2}} =\frac{\theta_{4}(Q^{2},t)}{\theta_{1}(Q^{4},t)^{2}}.
\]
\end{proposition}

Examining the first two terms in the $Q$ expansion we get the following
easy corollary:

\begin{corollary}\label{cor: c(-1,k) and c(0,k)}
The coefficients $c(a,k)$ for $a=-1,0$ are given by 
\[
c(-1,k) =
\begin{cases}
0,& -k\leq 0\\
-k,&k>0
\end{cases}
\quad \quad
c(0,k) =
\begin{cases}
0,& k< 0\\
1,&k=0\\
2k,&k>0
\end{cases}
\]
\end{corollary}

Applying the Jacobi triple product identity, we also get the following corollary.

\begin{corollary}\label{cor: sum c(a,k)Q^a t^k = product}
\[
\sum_{a=-1}^{\infty} \sum_{k\in \ZZ} c(a,k) Q^{a+1} t^{k} =
\frac{-t}{(1-t)^{2}} \prod_{n=1}^{\infty} \frac{(1-Q^{2n})(1-tQ^{2n-1})(1-t^{-1}Q^{2n-1})}{(1-Q^{4n})^{2}(1-tQ^{4n})^{2}(1-t^{-1}Q^{4n})^{2}}
\]
\end{corollary}

\subsection{Proof of Proposition~\ref{prop: generating function for c(a,k) =
c(4n-l^2,k) =c(n,l,k)}}

From the product formula for $\theta_{1}(q,t)$ and the fact that
$\theta_{1}(q,t^{-1})=-\theta_{1}(q,t)$, we see we may write
\[
\theta_{1}(q,t)^{-1}\theta_{1}(q,t^{-1})^{-1} = -\theta_{1}(q,t)^{-2}
= q^{-\frac{1}{4}}\frac{t}{(1-t)^{2}}\sum_{i=0}^{\infty} \delta_{i}(t) q^{i}
\]
where $\delta_{0}=1$ and $\delta_{i}(t)\in \ZZ [t,t^{-1}]$. Then
\[
\Ell_{q,y}(\CC^{2} ,t) = q^{-\frac{1}{4}}\frac{t}{(1-t)^{2}}
\sum_{i=0}^{\infty} \delta_{i}(t)q^{i} \sum_{n,m\in \ZZ} \, \,
q^{\half (n+\half )^{2}+\half (m+\half )^{2}} (-yt)^{n+\half}
(-yt^{-1})^{m+\half}. 
\]
If we let $l=m+n+1$ and $b=n-m$ and we note that then $l\equiv  b+1$
mod 2 and
$\half (n+\half )^{2}+\half (m+\half )^{2} = \frac{1}{4}(l^{2}+b^{2})$
we get
\[
\Ell_{q,y}(\CC^{2} ,t) = \frac{-t}{(1-t)^{2}} \sum_{i=0}^{\infty}
\delta_{i}(t)q^{i} \sum_{\begin{smallmatrix} l,b\in \ZZ \\
l\equiv b+1\mod 2 \end{smallmatrix}} q^{\frac{1}{4}(l^{2}+b^{2}-1)}
y^{l}(-t)^{b} 
\]
The terms with $q^{n}y^{l}$ in the above sum occur when
$n=i+\frac{1}{4}(l^{2}+b^{2}-1)$ and thus when $4n-l^{2} =
4i+b^{2}-1\geq -1$. Therefore the $q^{n}y^{l}$ coefficient of
$\Ell_{(q,y)}(\CC^{2} ,t)$ only depends on $a=4n-l^{2}$ and is zero if $a<-1$: 
\begin{align*}
\operatorname{Coef}_{q^{n}y^{l}}\left[\Ell_{(q,y)}(\CC^{2} ,t) \right] &= \sum_{k\in \ZZ } c(n,l,k)t^{k}\\
&=\sum_{k\in \ZZ }c(4n-l^{2},k)t^{k}
\end{align*}
where
\begin{align*}
\sum_{a=-1}^{\infty} \sum_{k\in \ZZ  } c(a,k)Q^{a}t^{k} &=
\frac{-t}{(1-t)^{2}} \sum_{i=0}^{\infty} \sum_{b\in \ZZ}
Q^{4i+b^{2}-1} (-t)^{b}\delta_{i}(t) \\
&=\frac{-tQ^{-1}}{(1-t)^{2}} \sum_{i=0}^{\infty} \delta_{i}(t)Q^{4i}
\sum_{b\in \ZZ} Q^{b^{2}} (-t)^{b} \\
&=\theta_{1}(Q^{4},t)^{-2} \sum_{b\in \ZZ} Q^{b^{2}} (-t)^{b}\\
&= \frac{\sum_{k\in \ZZ}Q^{k^{2}}(-t)^{k}}{\left(\sum_{k\in \ZZ
+\half} Q^{2k^{2}}(-t)^{k} \right)^{2}} \quad .
\end{align*}

\qed 
  
\section{Computing the partition function}\label{sec: computing the
part fnc}

\subsection{Overview}
Our basic strategy for computing the partition function
$Z^{\DTss}_{\Gamma}(X_{\ban})$ is to stratify the Hilbert scheme, where
each strata parameterizes subschemes supported on a prescribed set of
fibers of the map $\pi :X\to \PP^{1}$. Each stratum admits an action
by the Mordell-Weil group of sections on an infinitesimal neighborhood
of a fiber.

We use these group actions to reduce the $\nu$-weighted Euler
characteristic computation to the fixed point set of the group
actions. The fixed points correspond to subschemes which are supported
on an infinitesimal neighborhood of the banana configurations and
which are formally locally given by monomial ideals. We can count
these fixed points (weighted by $\nu$) using a technique adapted from
\cite{MNOP1} to our setting. The outcome is an expression for the
partition function in terms of the topological vertex.

The vertex expression we get has essentially been computed in the
physics literature. The bulk of the computation is done by Hollowood,
Iqbal, and Vafa \cite{Hollowood-Iqbal-Vafa} using geometric
engineering. Their derivation requires a certain geometrically
motivated combinatorial conjecture. The conjecture was a special case
of the Equivariant Crepant Resolution Conjecture for Elliptic Genera
(also called the equivariant DMVV conjecture after
Dijkgraaf-Moore-Verlinde-Verlinde \cite{DMVV}), which was subsequently
proven by Borisov-Libgobner and Waelder
\cite{Borisov-Libgober,Waelder} using motivic integration. We present
a mathematically self-contained version of the Hollowood-Iqbal-Vafa
derivation in \S~\ref{sec: vertex calculation.}.

\subsection{Preliminaries on Notation and Euler Characteristics}
For any scheme $Y$ over $\CC$, let $e(Y)$ be the topological Euler
characteristic of $Y$ in the complex analytic topology. Note that
Euler characteristic is independent of any nilpotent structures,
i.e. $e(Y)=e(Y_{red})$.

For any constructible function $\mu :Y\to \ZZ$, let
\[
e(Y,\mu ) = \sum_{k\in \ZZ} k\cdot e(\mu^{-1}(k))
\]

be the $\mu$-weighted Euler characteristic of $Y$.

We will need the following standard facts about Euler
characterisitics.
\begin{itemize}
\item Euler characteristic defines a ring homomorphism
$e:K_{0}(\Var_{\CC})\to \ZZ$, i.e. it is additive under the
decomposition of a scheme into an open set and its complement, and it
is multiplicative on Cartesian products.
\item For any constructible morphism\footnote{A constructible morphism
is a map which is regular on each piece of a decomposition of its
domain into locally closed subsets.} $f:Y\to Z$ we have (see
\cite{MacPherson-Annals74})
\begin{equation}\label{eqn: e(Y,mu)=e(Z,f_*(mu))}
e(Y,\mu ) = e(Z,f_{*}\mu )
\end{equation}
where $f_{*}\mu $ is the constructible function given by 
\[
(f_{*}\mu )(x) = e(f^{-1}(x),\mu ).
\]
\item If $\CC^{*}$ acts on a scheme $Y$ with fixed point locus
$Y^{\CC^{*}}\subset Y$, then (see \cite{Bialynicki-Birula})
\[
e(Y) = e(Y^{\CC^{*}}). 
\]
\end{itemize}

We will also use the following
\begin{lemma}\label{lem: if e(V)=0 if G acts on V and e(any orbit)=0}
Let $G$ be an algebraic group acting on $V$, a scheme (of finite type
over $\CC$). Let $\mu$ be a $G$-invariant constructible function on
$V$. Suppose that each $G$-orbit has zero Euler characteristic,
i.e. $e(O_{x})=0$  for all $x\in V$. Then $e(V,\mu )=0$. 
\end{lemma}

\begin{proof}
By a theorem of Rosenlicht \cite[Thm~2]{Rosenlicht}, we have that for the action
of any algebraic group $G$ on any variety $V$, there is a dense open
set $U\subset V$ and a morphism $\tau :U\to W$ to a variety $W$, such
that for all $x\in U$, $\tau ^{-1}(x)$ is a $G$-orbit. By iteratively
applying the same theorem to $V\setminus U$, we obtain a locally
closed $G$-equivariant stratification $V=\cup_{\alpha}U_{\alpha}$ such
that the $G$-action on each strata has a geometric quotient
\[
\tau_{\alpha}:U_{\alpha}\to W_{\alpha}\,\,\, .
\]
Suppose that every $G$-orbit has zero Euler characterisic,
$e(O_{x})=0$. Let $\mu$ be a $G$-invariant constructible function on $V$,
then
\begin{align*}
e(V,\mu ) & = \sum_{\alpha} e(U_{\alpha},\mu )\\
     & = \sum_{\alpha} e(W_{\alpha} , (\tau_{\alpha})_{*}(\mu ))\\
&=\sum_{\alpha}e(W_{\alpha},0)\\
&=0.
\end{align*}
This proves the lemma in the case where $V$ is a variety. For $V$ a
general scheme of finite type over $\CC$, we can easily construct a
$G$-equivariant stratification of the reduced space of $V$
\[
V_{red}= V_{1} \cup \dotsb \cup V_{N}
\]
such that each strata $V_{i}$ is a variety.
Then 
\[
e(V,\mu )=e(V_{red},\mu )=\sum_{i=1}^{N} e(V_{i},\mu ) = 0.
\]
\end{proof}

We will use the notation
\[
\Hilb^{\bullet}(X_{\ban}) = \sum_{\dvec ,n}
\Hilb^{\betadvec\,\, ,n}(X_{\ban})\, Q_{1}^{d_{1}} Q_{2}^{d_{2}}
Q_{3}^{d_{3}} (-p)^{n}
\]
where we regard the right hand side as a formal series in
$Q_{1},Q_{2},Q_{3}$, Laurent in $p$, having coefficients in
$K_{0}(\Var_{\CC})$, the Grothendieck group of varieties\footnote{We
note that $\Hilb^{0,0}(X_{\ban}$) is a single point corresponding to
the empty subscheme. Thus the constant term of the series
$\Hilb^{\bullet} (X_{\ban})$ is 1.}. We extend the operations in the
Grothendieck group (addition, multiplication, and Euler
characteristic) to the series in the obvious way. So for example, in
this notation, the partition function is given by:
\[
Z^{\DTss}_{\Gamma}(X_{\ban}) = e\left(\Hilb^{\bullet}(X_{\ban}),\nu  \right).
\]

We will apply the above notation more generally. Whenever we have any
collection $\{ A^{\dvec ,n} \}$ elements of a set (for example $\ZZ$
or $K_{0}(\Var_{\CC})$) indexed by $(\dvec ,n)$, we will write:

\[
A^{\bullet} = \sum_{\dvec ,n} A^{\dvec ,n} \, Q_{1}^{d_{1}} Q_{2}^{d_{2}}
Q_{3}^{d_{3}} (-p)^{n}.
\]

\subsection{Pushing forward the $\nu$-weighted Euler characteristic measure.}

Let 
\[
\Conf^{k} \PP^{1} =\Sym^{k}\PP^{1} \setminus \Delta
\]
be the configuration space of $k$ unordered, distinct points in
$\PP^{1}$, i.e. the $k$th symmetric product of $\PP^{1}$ with the big
diagonal deleted. Let
\[
\Conf \PP^{1} = \bigcup_{k} \Conf^{k}\PP^{1}.
\]

We define a constructible morphism by
\begin{align*}
\rho^{\dvec ,n} : \Hilb^{\beta_{\dvec}\, ,n}(X_{\ban }) & \to \Conf \PP^{1}\\
Z&\mapsto \operatorname{Supp}(\pi_{*}\OO_{Z})
\end{align*}
Then using equation \eqref{eqn: e(Y,mu)=e(Z,f_*(mu))} and our bullet notation, we have
\begin{align*}
Z_{\Gamma }^{\DTss }(X_{\ban })& = e\left(\Hilb^{\bullet}(X_{\ban }),\nu \right)\\
& = e\left( \Conf \PP^{1},\rho^{\bullet}_{*} \nu \right). 
\end{align*}
Note that the constructible function $\rho^{\bullet}_{*}\nu$ takes
values in $\ZZ [[Q_{1},Q_{2},Q_{3}]]((p)).$ 

\subsection{Subschemes supported in an infinitesimal neighborhood of a
fiber.}

From here on out, we will write $X$ for $X_{\ban}$ for the sake of
brevity when there is no ambiguity. We will also suppress the
superscripts $\betacomman$ when they are understood and/or
unimportant.

\begin{definition}\label{def: Fy, Fy(k), and FyHat}
Let $F_{y}$
be the fiber of $\pi :X\to \PP^{1}$ over $y\in \PP^{1}$. Let
$F_{y}^{(k)}$ be the $k$-order infinitesimal neighborhood of $F_{y}$
in $X$ and let $\Fhat_{y}$ be the formal neighborhood of $F_{y}$ in
$X$. We define 
\[
\Hilb (\Fhat_{y} )\subset \Hilb (X)
\]
to be the locally closed subscheme of $\Hilb (X)$
parameterizing subschemes $Z\subset X$ supported on $\Fhat_{y}$,
i.e. subschemes which are set theoretically, but necessarily scheme
theoretically, supported in $F_{y}$.  Finally, we define 
\[
\HilbHat  (\Fhat _{y} )
\]
to be the formal neighborhood of
$\Hilb (\Fhat_{y})$ in $\Hilb (X)$.
\end{definition}

The value of the Behrend function $\nu :\Hilb (X)\to \ZZ$
at a point only depends on the formal neighborhood of that point
\cite{Jiang-Motivic-Milnor-fibre} inside the Hilbert
scheme. Therefore $\nu$ restricted to $\Hilb (\Fhat_{y})$
is completely determined by the formal scheme
$\HilbHat (\Fhat_{y})$.

\begin{remark}
The closed points of $\Hilb (\Fhat_{y})$ and $\HilbHat (\Fhat_{y})$
are of course the same; they correspond to subschemes $Z\subset X$
such that $Z_{red}\subset F_{y}$. However, $\Hilb (\Fhat_{y})$ and
$\HilbHat (\Fhat_{y})$ classify different families of subschemes. For
example, the infinitesimal deformations of $Z$ parameterized by $\Hilb
(\Fhat_{y})$ must preserve the closed condition $Z_{red}\subset
F_{y}$, whereas $\HilbHat (\Fhat_{y})$ includes \emph{all}
infinitesimal deformations of $Z$, even those which (infinitesimally)
violate the condition $Z_{red}\subset X$.
\end{remark}

\subsection{Mordell-Weil groups and actions on $\Hilb
(\Fhat_{y})$.}\label{subsec: Mordell weil gps and actions}

Let 
\[
X^{\circ}=\left\{x\in X \text{ such that $\pi :X \to \PP^{1}$ is
smooth at $x$} \right\}.
\]
In other words, $X^{\circ}$ is $X$ with the 12 banana configurations
removed. Then, after fixing a section $s_{0}:\PP^{1}\to X^{\circ}$, 
\[
\pi^{\circ } :X^{\circ}\to \PP^{1}
\]
has the structure of an Abelian group scheme
over $\PP^{1}$. Let $F_{y}^{\circ}=X^{\circ}\cap F_{y}$ be the group
of the fiber over $y$. Let 
\[
\{x_{1},\dotsc ,x_{12} \}\subset \PP^{1}
\]
be the points with singular fiber. Then for $y\not \in \{x_{1},\dotsc
x_{12} \}$, $F_{y}^{\circ} \cong E\times E$, where $E$ is an elliptic
curve, and for $y\in \{x_{1},\dotsc ,x_{12} \}$, $F_{y}^{\circ }\cong
\CC^{*}\times \CC^{*}$.

\begin{definition}\label{defn: MW groups}
Let $\pi^{\circ}:X_{B}^{\circ}\to  B$ and $\pi :X_{B}\to B$ denote the
$B$-schemes obtained from $X^{\circ}$ and $X$ by some base change $B\to
\PP^{1}$. We define $\MW (B)$ be the Mordell-Weil group of sections of
$\pi^{0}$ over $B$.
\end{definition}

If $B\to \Spec \CC$ is finite, then $\MW
(B)$ is the Weil restriction of $X^{\circ}_{B}\to B$ with respect to
$B\to \Spec \CC$ and is thus an algebraic group over $\CC$, see
\cite[\S 7 Thm 4]{Bosch-Lutkebohmert-Raynaud}.

We get an action of $\MW (B)$ on $X_{B}$ defined as follows
(c.f. \cite[\S 7 Thm 6]{Bosch-Lutkebohmert-Raynaud}).  The group
scheme structure morphism $X^{\circ}_{B}\times_{B}X^{\circ}_{B}\to
X^{\circ}_{B}$ extends (see \cite{Deligne-Rapoport}) to a morphism
\[
\begin{tikzcd}
&X_{B}^{\circ}\times_{B} X_{B} \arrow[r,"+"]&  X_{B}.
\end{tikzcd}
\]
We define the $\MW (B)$ action on $X_{B}$ by the composition
\[
\begin{tikzcd}
\MW (B) \times X_{B}  \arrow[r,"\phi "] & X_{B}^{\circ }\times_{B}X_{B} \arrow[r,"+"]&  X_{B}
\end{tikzcd}
\]
where $\phi ={(ev\circ(Id\times\pi),pr_{X_{B}})}$,
\[
ev:\MW (B) \times B\to X_{B}
\]
is the tautological evaluation map $(s,y)\mapsto s(y)$, and 
\[
pr_{X_{B}}:\MW (B)\times X_{B}\to X_{B}
\]
is projection. $\phi$ is well defined because $\pi \circ ev\circ
(Id\times \pi )=\pi \circ pr_{X_{B}}$.

Concretely, if $x\in X_{B}$ is a point then the action of $s$ on $x$
is given by translation: $(s,x)\mapsto s(\pi (x))+x$.

Let $\Delta_{y}^{(k)} \subset \PP^{1}$ be the $k$th order thickening
of $y\in \PP^{1}$. Since
\[
\Delta_{y}^{(k)}\cong \Spec \CC[\epsilon ]/\epsilon^{k+1}\to \Spec \CC
\]
is finite, $\MW (\Delta_{y}^{(k)})$ is an algebraic group. Restriction
of a section to the closed fiber expresses $\MW (\Delta_{y}^{(k)})$ as
an extension of $F^{\circ}_{y}$ by a vector group of some dimension
$D=D(k)$:
\begin{equation}\label{eqn: exact sequence for MW(k)}
\begin{tikzcd}
0\arrow{r} &\CC^{D} \arrow{r} &
\MW(\Delta _{y}^{(k)}) \arrow[r,"r"] &
F^{\circ}_{y}  \arrow{r}
&0
\end{tikzcd}
\end{equation}

We write $F_{k}^{(k)}$ for $X_{\Delta_{y}^{(k)}}$, i.e. the $k$th
order thickening of the fiber $F_{y}$. The action of $\MW
(\Delta_{y}^{(k)})$ on $F_{y}^{(k)}$ is compatible with the
restriction homomorphisms
\[
\MW(\Delta ^{(k+1)}_{y}) \to \MW(\Delta ^{(k)}_{y})
\]
and the inclusions 
\[
F_{y}^{(k)}\subset F_{y}^{(k+1)}.
\]
Let
\[
\MW^{(\infty )}_{y} = \lim_{\stackrel{\leftarrow}{k}} \MW(\Delta ^{(k)}_{y})
\]
be the inverse limit group. Then by construction, $\MW^{(\infty
)}_{y}$ acts on $\Fhat_{y}$, and this induces an action of
$\MW^{(\infty )}_{y}$ on $\Hilb (\Fhat_{y})$ and on $\HilbHat
(\Fhat_{y})$. Note that 
\[
\Hilb^{\betacomman}(\Fhat_{y}) \subset \Hilb^{\betacomman}(F_{y}^{(N)})
\]
for some large $N=N(\betacomman)$ which depends on $\beta$ and
$n$. Therefore, $\Hilb^{\betacomman}(\Fhat_{y})$ is acted on by the
algebraic group $\MW(\Delta ^{(N)}_{y})$. However, $\MW(\Delta
^{(N)}_{y})$ does not act on $\HilbHat^{\betacomman}(\Fhat_{y})$ since
this includes all infinitesimal deformations of any $Z\subset
\Fhat_{y}$, which can involve finite neighborhoods of arbitrarily big
orders. Thus only the pro-algebraic group $\MW^{(\infty )}_{y}$ acts
on $\HilbHat^{\betacomman}(\Fhat_{y})$.

Since the Behrend function $\nu :\Hilb (\Fhat_{y})\to \ZZ $ is determined by
$\HilbHat (\Fhat_{y})$, the action of $\MW^{(\infty )}_{y}$ must preserve
$\nu$.  Consequently we have:
\begin{lemma}\label{lem: MW(k) action preserves nu}
The Behrend function $\nu$ is invariant under the action of
$\MW(\Delta ^{(N)}_{y})$ on $\Hilb^{\betacomman}(\Fhat_{y})$.
\end{lemma}

In general, we do not know if the sequence \eqref{eqn: exact sequence for
MW(k)}  splits.  However, if $F_{y}$ is a singular fiber so that
$F^{\circ}_{y}\cong \CC^{*}\times \CC^{*}$, then \eqref{eqn: exact
sequence for MW(k)} must split:

\[
\begin{tikzcd}
0\arrow{r} &\CC^{D} \arrow{r} &
\MW(\Delta _{y}^{(k)})
\arrow[r,"r"'] & \arrow[l,bend right = 15]\arrow{r}
\CC^{*}\times \CC^{*}
& 1
\end{tikzcd}
\]
because then $\MW(\Delta _{y}^{(k)})$ is an affine commutative algebraic
group over $\CC$ and hence a product of a torus and a vector space
group.  Consequently, we have the following
\begin{lemma}\label{lem: C*xC* acts on Hilb(Fhatsing)}
The group $\CC^{*}\times \CC^{*}$ acts on $\Hilb (\Fhat_{\sing})$
where $\Fhat_{\sing}$ is any singular fiber. Moreover, the action
extends to  $\HilbHat (\Fhat_{\sing})$ and thus it
preserves the Behrend function $\nu$. 
\end{lemma}

\subsection{Reduction of the computation to the singular fibers.}

\begin{lemma}\label{lem: e(Hilb(Fhat),nu)=0 for smooth fibers}
Let $F_{y}$ be a smooth fiber, then
\[
e(\Hilb(\Fhat_{y}),\nu ) = 0.
\]
\end{lemma}

\begin{proof}
By Lemma~\ref{lem: MW(k) action preserves nu}, $\MW =\MW(\Delta
_{y}^{N(\beta ,n)})$ acts on $\Hilb^{\beta_{\dvec}\, ,n}(\Fhat_{y})$
preserving $\nu$. Moreover, the action is fixed point free since the
group $\MW$ acts transitively on $\Fhat_{y}$. Let $\Stab_{x}$ be the
$\MW$-stabilizer of $x\in \Hilb (\Fhat_{y})$ and let $Z_{x}\subset
\Fhat_{y}$ be the subscheme corresponding to $x$. Then the image
$r(\Stab_{x})$ is a proper subgroup of $F_{y}$ since there is some
element in $F_{y}$ which does not preserve
$\operatorname{Supp}(Z_{x})$. Then the orbit of $x$, $O_{x}=\MW
/\Stab_{x}$ is an Abelian group fiting into the following exact
sequence
\[
0\longrightarrow \frac{\CC^{D}}{\Stab_{x}\cap \,\,\CC^{D}} \longrightarrow O_{x}\longrightarrow \frac{F_{y}}{r(\Stab_{x})}\longrightarrow 0.
\]
Therefore, $O_{x}$ is given as the total space of a smooth fibration
with base $F_{y}/r(\Stab_{x})$, a positive dimensional Abelian
variety. It follows that $e(O_{x})=0$.  We then can apply
Lemma~\ref{lem: if e(V)=0 if G acts on V and e(any orbit)=0} with
$G=\MW $, $V=\Hilb (\Fhat_{y})$, and $\mu =\nu $ to complete the proof.

\end{proof}

Recall that $Z_{\Gamma}^{\DTss} (X) = e(\Conf
\PP^{1},\rho^{\bullet}_{*}\nu )$. To compute $\rho^{\bullet}_{*}\nu$, 
we note that the preimage
\begin{align*}
(\rho^{\bullet} )^{-1}(\{y_{1},\dotsc ,y_{k} \})&  = \Hilb^{\bullet }(\Fhat_{y_{1}}\cup \dotsb \cup \Fhat_{y_{k}}) \\
&= \prod_{i=1}^{k}\Hilb^{\bullet}(\Fhat_{y_{i}})
\end{align*}
where the product takes place in the ring
$K_{0}(\Var_{\CC})[[Q_{1},Q_{2},Q_{3}]]((p))$.

Lemma~\ref{lem: e(Hilb(Fhat),nu)=0 for smooth fibers} implies that
\[
(\rho^{\bullet}_{*}\nu)(\{y_{1},\dotsc ,y_{k} \}) = 0\quad\text{ if }\quad  \{y_{1},\dotsc ,y_{k} \}\not \subset \{x_{1},\dotsc ,x_{12} \}
\]
where $F_{x_{1}},\dotsc ,F_{x_{12}}$ are the 12 singular fibers.

In other words, we have shown that the constructible function
$\rho^{*}_{\bullet}\nu$ is supported on $\Conf (\{x_{1},\dotsc ,x_{12}
\})\subset \Conf \PP^{1}$. Therefore
\begin{align*}
Z^{\DTss}_{\Gamma}(X) &= e(\Conf \PP^{1},\nu )\\
&=e(\Conf (\{x_{1},\dotsc ,x_{12} \}),\rho^{*}_{\bullet}) \\
&= e(\Hilb^{\bullet} (\Fhat_{x_{1}} \cup \dotsb \cup  \Fhat_{x_{12}}),\nu )\\
&=\prod_{i=1}^{12}e(\Hilb^{\bullet}(\Fhat_{x_{i}}),\nu ).\\
\end{align*}
Since all the formal neighborhoods $\Fhat_{x_{i}}$ are isomorphic, we
will write $\Fhat_{\ban}$ for this formal scheme, i.e. the formal
neighborhood of $F_{\ban}$, a fiber containing a banana
configuration. In conclusion we have
\begin{equation}\label{eqn: Z = (Z(Fsing))^12}
Z^{\DTss}_{\Gamma}(X_{\ban}) = e\left(\HilbHat^{\bullet}(\Fhat_{\ban } ),\nu \right)^{12}.
\end{equation}
Note that here $\nu$ is the restriction of the Behrend function on
$\Hilb (X)$ to $\Hilb (\Fhat_{\ban})$, but since this is determined
by $\HilbHat (\Fhat_{\ban})$, we may regard $\nu$ as the Behrend
function of the formal scheme $\HilbHat (\Fhat_{\ban})$. We will
write
\[
Z^{\DTss}(\Fhat_{\ban}) =  e\left(\HilbHat^{\bullet}(\Fhat_{\ban } ),\nu \right).
\]

\begin{remark}\label{rem: proof applies to general bananafolds}
For more general banana manifolds, the same proof shows that
equation~\ref{eqn: Z = (Z(Fsing))^12} holds with the 12 replaced by
$N$, the number of singular fibers (i.e. the number of banana
configurations).
\end{remark}

\subsection{Reduction to $\CC^{*}\times \CC^{*}$-fixed
subschemes.}\label{subsec: reduction to C*-fixed subschemes}

The $\CC^{*}\times \CC^{*}$ action on $\Fhat_{\ban}$ preserves the
canonical class by construction. In particular, the action of
$\CC^{*}\times \CC^{*}$ is compatible with the symmetric obstruction
theory \cite{Behrend-Fantechi08} of $\Hilb (X)$, restricted to
$\HilbHat (X)$. We note that in order to work in this formal setting,
we must use the symmetric obstruction theory associated to the formal
moduli space (defined by Jiang in
\cite{Jiang-Motivic-Milnor-fibre}). The result of Behrend and Fantechi \cite{Behrend-Fantechi08}
is that for an isolated fixed point $P\in \HilbHat (\Fhat_{\ban})$,
the value of the Behrend function is given by
\[
\nu (P) = (-1)^{\dim T_{P}}
\] 
where
\[
T_{P} = T_{P} \HilbHat (\Fhat_{\ban})
\]
is the Zariski tangent space of $P\in  \HilbHat (\Fhat_{\ban})$. We
note that 
\[
T_{P}= \Ext^{1}_{0}(I_{Z_{P}},I_{Z_{P}})
\]
where $Z_{P}$ is the subscheme associated to $P$ (see \cite{MNOP1}).

We will see that the fixed points of the action are
isolated. It then follows that
\begin{align*}
Z^{\DTss}(\Fhat_{\ban}) &= e\left(\HilbHat^{\bullet}(\Fhat_{\ban}),\nu \right)\\
 &=e\left(\HilbHat^{\bullet}(\Fhat_{\ban})^{\CC^{*}\times \CC^{*}},\nu \right) \\
&=\sum_{P\in \HilbHat (\Fhat_{\ban})^{\CC^{*}\times \CC^{*}}}
(-1)^{\dim T_{P}} Q_{1}^{d_{1}(P)}
Q_{2}^{d_{2}(P)} Q_{3}^{d_{3}(P)} (-p)^{\chi (\OO_{Z_{P}})}
\end{align*}
where $d_{i}(P)$ is defined by:
\[
[Z_{P}] = d_{1}(P)C_{1}+ d_{2}(P)C_{2}+ d_{3}(P)C_{3}.
\]

\begin{proposition}\label{prop: fixed points of Hilb are given by R1,R2,R3,pi1,pi2}
The $\CC^{*}\times \CC^{*}$ fixed points $P\in \HilbHat
(\Fhat_{\ban})^{\CC^{*}\times \CC^{*}}$ are isolated and the fixed
point set is in bijective correspondence with the set
$\{R_{1},R_{2},R_{3},\pi_{1},\pi_{2} \}$ where $R_{1},R_{2},R_{3}$ is
a triple of 2D partitions and $\pi_{1},\pi_{2}$ is a pair of 3D
partitions asymptotic to $(R_{1},R_{2},R_{3})$ and
$(R'_{1},R'_{2},R'_{3})$ respectively\footnote{see for example
\cite[Def.~1]{Bryan-Kool-Young} for a definition of 3D partitions
asymptotic to a triple of 2D partitions. Also, $A'$ denotes the
partition conjugate to $A$, see \S~\ref{sec: vertex calculation.}.}.
Moreover, the discrete invariants of the corresponding subscheme are
given by
\begin{align*}
d_{i}(P) & = |R_{i}|\\
\chi (\OO_{Z_{P}}) &= |\pi_{1}|+|\pi_{2}| + \half \sum_{i=1}^{3}
||R_{i}||^{2} + ||R'_{i}||^{2} \\
(-1)^{\dim T_{P}} & = (-1)^{\chi (\OO_{Z_{P}}) +|R_{1}|+|R_{2}|+|R_{3}|} 
\end{align*}
where $|R_{i}|$ is the size of a partition, $|\pi_{i}|$ is the
normalized volume \cite{Bryan-Kool-Young,Ok-Re-Va}, and
$||R_{i}||^{2}$ denotes the sum of the squares of the parts
(c.f. \S~\ref{sec: vertex calculation.}).
\end{proposition}

This proposition will be proved in the next section, but we first use
it to finish the computation $Z(\Fhat_{\ban})$. First recall that
the topological vertex is defined to be the generating function
\[
\Vsf_{R_{1}R_{2}R_{3}}(p) = \sum_{\pi} p^{|\pi |}
\]
where the sum is over all 3D-partitions asymptotic to
$(R_{1},R_{2},R_{3})$ \cite[Def.~2]{Bryan-Kool-Young},
\cite{Ok-Re-Va}. Then by the proposition  we obtain
\begin{align*}
Z(\Fhat_{\ban}) =& \sum_{R_{1}R_{2}R_{3}} (-Q_{1})^{|R_{1}|}
(-Q_{2})^{|R_{2}|} (-Q_{3})^{|R_{3}|} p^{\half \sum_{i=1}^{3}
||R_{i}||^{2}+ ||R'_{i}||^{2}} \Vsf_{R_{1}R_{2}R_{3}}(p)  \Vsf_{R'_{1}R'_{2}R'_{3}}(p) \\
 = & M(p)^{2}  \sum_{R_{1}R_{2}R_{3}} (-Q_{1})^{|R_{1}|}
(-Q_{2})^{|R_{2}|} (-Q_{3})^{|R_{3}|}  \Vtildesf_{R_{1}R_{2}R_{3}}(p)  \Vtildesf_{R'_{1}R'_{2}R'_{3}}(p)
\end{align*}
where 
\[
\Vtildesf_{R_{1}R_{2}R_{3}} = M(p)^{-1} p^{\half (
||R_{1}||^{2}+||R'_{2}||^{2}+ ||R_{3}||^{2})} V_{R_{1}R_{2}R_{3}}
\]
is the normalized vertex. The main result of \S~\ref{sec: vertex
calculation.}  (Theorem~\ref{thm: vertex computation}) then asserts:

\[
 Z(\Fhat_{\ban})    = \prod_{ a_{i}\geq 0} \prod_{\begin{smallmatrix}k\in \ZZ ,\\
k>0\text{ if }\avec=
\zerovec \end{smallmatrix}}
\left(1-Q_{1}^{a_{1}}Q_{2}^{a_{2}}Q_{3}^{a_{3}}p^{k} \right)^{-c(||\avec||,k)}.
\]

This formula, along with equation~\eqref{eqn: Z = (Z(Fsing))^12} then
proves Theorem~\ref{thm: formula for Zban}. 

\subsection{Analysis of $\CC^{*}\times \CC^{*}$-fixed
subschemes.}\label{subsec: analysis of fixed subschemes}

The goal of this section is to prove Propostion~\ref{prop: fixed
points of Hilb are given by R1,R2,R3,pi1,pi2}. 

To analyze the action of $\CC^{*}\times \CC^{*}$ on $\Fhat_{\ban}$ and
determine the $\CC^{*}\times \CC^{*}$ invariant subschemes, we obtain
an explicit toric description of $\Fhat_{\ban}$ using the following
proposition. Informally, the proposition says that
\begin{enumerate}
\item $F_{\ban}$ is obtained from $\Bl (\PP^{1}\times \PP^{1})$, the
blow up of $\PP^{1}\times \PP^{1}$ at the points $(0,0),(\infty ,
\infty )$, by gluing the three pairs of disjoint boundary divisors to
each other in normal crossings. 
\item $\Fhat_{\ban}$ is obtained from $\widehat{\Bl (\PP^{1}\times
\PP^{1})}$, the formal neighborhood of $\Bl (\PP^{1}\times \PP^{1})$
inside the total space of its canonical bundle, by an etal\'e gluing
(which restricts to the normal crossing gluing of $\Bl (\PP^{1}\times
\PP^{1})$).
\item The gluing is $\CC^{*}\times \CC^{*}$ equivariant.
\end{enumerate}

Let $F_{\sing}\subset X_{\sing}$ be the image of $F_{\ban}$ under the
blowup $X_{\ban}\to X_{\sing}$. Let 
\[
\sigma :F^{\norm}_{\sing}\to F_{\sing},\quad \tau :F^{\norm}_{\ban}\to
F_{\ban}
\]
be the normalizations. Define $\Fhat^{\norm}_{\ban}$ to be the formal
scheme given by the ringed space
$(F^{\norm}_{\ban},\tau^{*}\widehat{\OO}_{\Fhat_{\ban}}) $.
\begin{proposition}\label{prop: toric description of FhatBan}
 Then the following hold
\begin{enumerate}
\item $F^{\norm}_{\sing}\cong \PP^{1}\times \PP^{1}.$
\item $F^{\norm}_{\ban}\cong \Bl(\PP^{1}\times \PP^{1})$, the
blowup of $F^{\norm}_{\sing}$ at the two points $(0,0)$ and $(\infty
,\infty )$ (c.f. \cite[Fig~4]{Abouzaid-Auroux-Katzarkov}). 
\item $\Fhat^{\norm}_{\ban}\cong \widehat{\Bl( \PP^{1}\times
\PP^{1})}$, the formal neighborhood of $\Bl( \PP^{1}\times \PP^{1})$
viewed as the zero section in $\Tot (K_{ \Bl (\PP^{1}\times \PP^{1})})$,
the total space of its canonical bundle.
\item The induced map
$\Fhat^{\norm}_{\ban} \xrightarrow{\widetilde{\tau}} \Fhat_{\ban}$ is
etal\'e.  
\item All of the above maps are $\CC^{*}\times \CC^{*}$ equivariant.  
\end{enumerate}
\end{proposition}

These results are summarized in the following diagram:

\[
\hspace*{-2cm}
\begin{tikzcd}[column sep=1.2cm]
&[-17] \Fhat_{\ban}
\arrow[dl,hook, "
\overset{\text{\tiny formal}}{\text{\tiny neighborhood}}
"']
& \Fhat^{\norm}_{\ban}
\arrow[l,"\widehat{\tau}"',"\text{ etal\'e}"]
&[-25]\widehat{\Bl( \PP^{1}\times \PP^{1})}
\arrow[l,phantom,"\cong " description]
\arrow[dr,hook,"
\overset{\text{\tiny formal}}{\text{\tiny neighborhood}}
"]\\ 
X_{\ban}
\arrow[d]
&[-17]F_{\ban}
\arrow[d]
\arrow[u,hook]
\arrow[l,hook]
&F^{\norm}_{\ban}
\arrow[l,"{\tau}"',"\text{\tiny normalization}"]
\arrow[d]\arrow[u,hook]
&[-25]\Bl (\PP^{1}\times \PP^{1})
\arrow[u,hook]
\arrow[l,phantom,"\cong " description]
\arrow[d,"\text{\tiny Blow up $(0,0)$ and $(\infty ,\infty )$}"]
\arrow[r,hook]
&[-23]\Tot(K_{\Bl (\PP^{1}\times \PP^{1})})\\
X_{\sing}
&[-17]F_{\sing}
\arrow[l,hook]
&F^{\norm}_{\sing}
\arrow[l,"\sigma "',"\text{\tiny normalization}"]
&[-25]\PP^{1}\times \PP^{1}
\arrow[l,phantom,"\cong " description]
&
\end{tikzcd}
\]

\begin{proof}
The proof of this proposition is primarily based on computations in
formal local coordinates. We give the details in Section~\ref{subsec:
proof of prop about FhatBan}.
\end{proof}

Using the equivariant etal\'e morphism
\[
\hat{\tau} :\Fhat_{\ban}^{\norm}\to \Fhat_{\ban}
\]
provided by the lemma, we may study $\CC^{*}\times \CC^{*}$-invariant
subschemes of $\Fhat_{\ban}$ by studying $\CC^{*}\times
\CC^{*}$-invariant of $\Fhat^{\norm}_{\ban}$ satisfying the decent
condition. Since the only proper, invariant subschemes of 
\[
K_{\Bl} := \Tot \left(K_{\Bl (\PP^{1}\times \PP^{1})} \right)
\]
are supported on $\Fhat^{\norm}_{\ban}\subset K_{\Bl}$, we may
consider torus invariant subschemes of the toric threefold $K_{\Bl}$
and use the method of MNOP \cite{MNOP1} to count such subschemes ---
subject to the condition that such subschemes descend under the
etal\'e relation given by $\hat{\tau}$.

A torus invariant subscheme of a toric Calabi-Yau threefold is
determined by combinatorial data attached to its web diagram, a
trivalent planar graph (which includes non-compact edges). Namely,
each edge is labelled by a 2D partition and each vertex is labelled by
a 3D partition which is asymptotic to the three 2D partitions given by
the incident edges (see \cite{MNOP1} or also \cite[\S 3 and \S
B]{Bryan-Cadman-Young}). The web diagram of $K_{\Bl}$ is a hexagon
with 6 additional non-compact edges (the graph on the right in
Figure~\ref{Fig: momentum polytopes of Fsing and Fban}). The etal\'e
relation identifies the six vertices to two distinct vertices
corresponding to the points $\{p,q \}$ in the banana
configuration. The relation also identifies the edges to three
distinct edges corresponding to the curves $C_{1},C_{2},C_{3}$ in the
banana configuration.  Each $\CC^{*}\times \CC^{*}$ invariant
subscheme is thus determined by three 2D partitions
$R_{1},R_{2},R_{3}$ and two 3D partitions $\pi_{1},\pi_{2}$.  The
quantaties $d_{i}(P)$, $\chi (\OO_{Z_{P}})$, and $(-1)^{\dim T_{P}}$
are computed by MNOP in terms of the combinatorial data associated to
the fixed point $P$ for toric Calabi-Yau threefolds. Their computation
applies in our setting (of a formal Calabi-Yau threefold) as
well. Their proofs are based on computing with the Cech open cover
given by the $\CC ^{3}$ coordinate charts, but we can equally well use
the Cech cover obtained by
intersecting the toric open cover on $K_{\Bl}$ with the formal
neighborhood of the banana configuration, and then identifying via the
etal\'e relation. The proof that fixed points are isolated \cite[Lemma
6]{MNOP1}, the computation of the degree and Euler characteristic
\cite[\S 4.4]{MNOP1}, and the computation of the parity of the
dimensional of the tangent space \cite[Thm~2]{MNOP1} all work with
this more general Cech cover. 

Applying the formulas of MNOP is then straightforward. The formulas in
our Proposition~\ref{prop: fixed points of Hilb are given by
R1,R2,R3,pi1,pi2} follow directly from the formulas in
\cite[\S4.4]{MNOP1} (in particular Lemma 5), \cite[Theorem~2]{MNOP1},
and the fact that
\[
\sum_{(i,j)\in R} (i+j+1)\\= \half (||R||^{2} + ||R'||^{2}). 
\]

This completes the proof of Proposition~\ref{prop: fixed points of
Hilb are given by R1,R2,R3,pi1,pi2}.

\section{The vertex calculation}\label{sec: vertex calculation.}

Let
\[
\VertexTilde_{R_{1}R_{2}R_{3}}(p) = s_{R'_{3}}(p^{-\rho})\sum_{A}
s_{R_{1}'/A}(p^{-R_{3}-\rho}) s_{R_{2}/A}(p^{-R'_{3}-\rho})
\]
be the normalized vertex and let (see section~\ref{subsec: reduction to C*-fixed subschemes}) 
\[
Z(\Fhat_{\ban }) = M(p)^{2} \sum_{R_{1},R_{2},R_{3}} (-Q_{1})^{R_{1}}  (-Q_{2})^{R_{2}}
(-Q_{3})^{R_{3}} \VertexTilde_{R_{1}R_{2}R_{3}}(p)
\VertexTilde_{R'_{1}R'_{2}R'_{3}}(p)
\]
where $M(p) = \prod_{m=1}^{\infty}(1-p^{m})^{-m}$.

The purpose of this section is to prove the following.

\begin{theorem}\label{thm: vertex computation} 
\[
Z(\Fhat_{\ban})    = \prod_{ a_{i}\geq 0} \prod_{\begin{smallmatrix}k\in \ZZ ,\\
k>0\text{ if }\avec=
\zerovec \end{smallmatrix}}
\left(1-Q_{1}^{a_{1}}Q_{2}^{a_{2}}Q_{3}^{a_{3}}p^{k} \right)^{-c(||\avec||,k)}
\]
where $c(a,k)$ is given by Proposition~\ref{prop: generating function for c(a,k) =
c(4n-l^2,k) =c(n,l,k)}.
\end{theorem}

Most of this computation has previously appeared in the physics
literature under the guise of geometric engineering (in this case a
duality between a certain six dimensional $U(1)$ gauge theory and a
certain topological string theory). The main reference is
Hollowood-Iqbal-Vafa \cite{Hollowood-Iqbal-Vafa}. This calculation was
also studied in \cite{Li-Liu-Zhou}. These computations assumed an
equality between the generating function for the equivariant elliptic
genera of $\Hilb^{n}(\CC^{2})$, and the orbifold
equivariant elliptic genera of $\Sym^{n}(\CC^{2})$ which they call the
DMVV conjecture (after \cite{DMVV}). This is an
instance of the crepant resolution conjecture for elliptic genera,
proven in the compact case by Borisov-Libgober \cite{Borisov-Libgober}
and the equivariant case by Waelder \cite{Waelder}.

We give the derivation here in full detail. We have filled in some
minor details that are missing from the previous accounts and have
collected all the needed results in one place. 

\subsection{Overview of computation}

After collecting some standard Schur function identities and proving a
few new ones in the next subsection, we proceed to the main
computation. The basic structure is as follows. 
\begin{enumerate}
\item Writing 
\[
Z'_{\ban}=M(p)^{-2}Z(\Fhat_{\ban}) = \sum_{R} (-Q_{3})^{|R|} Z_{R}(Q_{1},Q_{2},p)
\]
we use a series of Schur function identities to simplify $Z_{R}$ and
write it as a product 
\[
Z_{R} = Z_{\inst}\cdot Z_{\hook,R}
\]
where $Z_{\inst}$ is a product of terms which do not depend on $R$ and
$Z_{\hook ,R}$ is also a product of terms which depend on
the hook-lengths of $R$.
\item We observe that after the change of variables
\begin{equation}\label{eqn: variable change Qs to q,y,etc}
Q_{1} = y, \quad Q_{2} = y^{-1}q,\quad Q_{3} = y^{-1}Q ,\quad p=t
\end{equation}
the product $Z_{\hook ,R}$ is exactly the contribution of  a
$\CC^{*}$-fixed point to the computation of
\[
\sum_{n=0}^{\infty}\Ell_{q,y}(\Hilb (\CC^{2}),t)\, Q^{n},
\]
the elliptic genera of the Hilbert schemes, via Atiyah-Bott
localization. Here $\lambda \in \CC^{*}$ acts on the factors of
$\CC^{2}$ with opposite weights and $t$ is the equivariant parameter.
\item An product formula for $\sum_{n=0}^{\infty}\Ell_{q,y}(\Hilb
(\CC^{2}),t)\, Q^{n}$ was conjectured by Dijkgraaf-Moore-Verlinde-Verlinde
\cite{DMVV}, and proven by Borisov-Libgobner and Waelder
\cite{Borisov-Libgober, Waelder}. Using that formula, substituting
back to the $Q_{1},Q_{2},Q_{3},p$ variables, and performing a few easy
manipulations, we arrive at Theorem~\ref{thm: formula for Zban}.
\end{enumerate}

\subsection{Notation and Schur function
identities}\label{subsec: Notation and schur function identities} 

We will use capital letters $R,A,B,C$, etc. to denote partitions. Via
its diagram, we
regard a partition $A$ as a finite subset of $\NN \times  \NN$ where if
$(i,j)\in A$ then $(i-1,j)\in A$ and $(i,j-1)\in A$. The \emph{rows}
or \emph{parts} of $A$ are the integers $A_{j} = \max \{i | (i,j)\in
A \}$. We use $'$ to denote the conjugate partition
\[
A' = \{(i,j):(j,i)\in A \},
\]
and we write
\[
|A| = \sum_{j}A_{j} ,\quad ||A||^{2} = \sum_{j} A_{j}^{2},
\]
For each $(i,j)\in A$ we define the \emph{hook length} :
\[
h_{ij}(A) = A_{i}+A_{j}'-i-j+1 . 
\]
We write $\square$ for the unique partition of size 1. 

We also use the notation
\[
M(u,p) = \prod_{m=1}^{\infty}(1-up^{m})^{-m}
\]
and the short hand $M(p)=M(1,p)$.

For a collection of variables $x=(x_{1},x_{2},\dotsc )$ and two
partitions $A$ and $B$ let $s_{A/B}(x) = s_{A/B}(x_{1},x_{2},\dotsc )$
denote the skew Schur function (see for example
\cite[\S~5]{Macdonald}). Let
\[
\rho  = \left(-\frac{1}{2},-\frac{3}{2},-\frac{5}{2},\dotsc  \right)
\]
so that for example $p^{-R-\rho}$ is notation for the variable list
\[
\left(p^{-R_{1}+\frac{1}{2}},p^{-R_{2}+\frac{3}{2}},\dotsc  \right).
\]

Okounkov-Reshetikhin-Vafa derived a formula for the topological vertex
in terms of skew Schur functions. Translating their formulas
\cite[3.20\& 3.21]{Ok-Re-Va} into our notation, we get:
\begin{equation}\label{eqn: ORV formula for vertex}
\Vsf_{R_{1}R_{2}R_{3}}(p) = M(p) \, p^{-\half (\| R_{1} \| ^{2}+\| R_{2}'
\| ^{2}+\| R_{3} \| ^{2})} \,\Vtildesf_{R_{1}R_{2}R_{3}}(p)
\end{equation}
where we've defined
\begin{equation}\label{eqn: defn of Vtilde}
\Vtildesf_{R_{1}R_{2}R_{3}}(p) = s_{R_{3} '}(\ptotheminusrho )\, \sum_{A}\, s_{R_{1}
'/A}(p^{-R_{3} -\rho})\cdot s_{R_{2} /A}(p^{-R_{3} '-\rho} ).
\end{equation}

We will need the following Schur function identities. We remark that
the Schur functions which appear are all Laurent expansions in $p$ or
in $p^{-1}$ of rational functions in $p$ and many of the identities
should be understood as equalities of rational functions.

From \cite[3.10]{Ok-Re-Va} we have\footnote{There is a typo in
equation 3.10 in \cite{Ok-Re-Va} --- the exponent on the right hand
side should be $-\nu '-\rho$.}
\begin{equation}\label{eqn: symmetric fnc involution identity }
s_{A/B}(p^{C+\rho}) = (-1)^{|A|-|B|} \, s_{A'/B'}(p^{-C'-\rho}). 
\end{equation}
From \cite[pg~45]{Macdonald}:
\[
s_{R}(1,p,p^{2},\dotsc ) = p^{n(R)} \prod_{(i,j)\in R}\frac{1}{1-p^{h_{ij}(R)}}
\]
where \cite[pg~3]{Macdonald} $n(R) = \half ||R'||^{2}-\half |R|$ and
so using the homogeneity of $s_{R}$ we see that
\[
s_{R}(\ptotheminusrho ) = p^{\half ||R'||^{2}}\prod_{i,j\in
R}\frac{1}{(1-p^{h_{ij}(R)})}
\]
and so we arrive at
\begin{equation}\label{eqn: sR(ptorho)sR'(ptorho)= hook product}
s_{R}(\ptotheminusrho )s_{R'}(\ptotheminusrho ) = (-1)^{|R|}
\prod_{i,j\in R} \frac{1}{(1-p^{h_{ij}(R)})(1-p^{-h_{ij}(R)})}.
\end{equation}

We will also need the following identity \cite[pg~93(2)]{Macdonald}
\begin{equation}\label{eqn: sum sR/B*sR'/A = prod(1+xiyj)sum sA'/C*sB'/C'}
\sum_{R_{1}} \,s_{R_{1}/B}(x) \, s_{R_{1}'/A}(y) =
\prod_{i,j=1}^{\infty} (1+x_{i}y_{j})\sum_{C} s_{A'/C}(x)s_{B'/C}(y).
\end{equation}

Finally we will need the following lemma.
\begin{lemma}\label{lem: infinite product = macmahon*finite hook product}
The following equalities hold as formal power series in $u$ whose
 coefficients are rational functions of $p$. 
\begin{align}
\prod_{j,k=1}^{\infty} \left(1-up^{R_{j}+R'_{k}-j-k+1} \right) &=M(u,p)^{-1}
\prod_{j,k\in R} \left(1-up^{h_{jk}(R)} \right)
\left(1-up^{-h_{jk}(R)} \right) \label{eqn: 1st formula in hook lemma}\\
 \prod_{j,k=1}^{\infty} \left(1-up^{-R_{j}+R_{k}+j-k} \right) &=M(u,p)
\prod_{j,k\in R} \left(1-up^{h_{jk}(R)} \right)^{-1}
\left(1-up^{-h_{jk}(R)} \right)^{-1}. \label{eqn: 2nd formula in hook lemma}
\end{align}
Note that the right hand side of the equations are invariant under $p
\leftrightarrow p^{-1}$ and so we get two more identities by replacing
$p$ by $p^{-1}$ on the left hand side of the above equations.
\end{lemma}
\begin{proof}
Moving the $M(u,p)$ to the left hand side of the first equation and
then taking the log, we get
\begin{align*}
&\quad \log \left(M(u,p) \prod_{j,k=1}^{\infty} \left(1-up^{R_{j}+R'_{k}-j-k+1} \right) \right)\\
=&\quad \log \left( \prod_{j,k=1}^{\infty} \left(1-up^{j+k-1}
\right)^{-1}  \right) + \log \left( \prod_{j,k=1}^{\infty} \left(1-up^{R_{j}+R'_{k}-j-k+1} \right)  \right) \\
=&\sum_{n\geq 1} \frac{u^{n}}{n} \sum_{j,k\geq 1}
\left(p^{n(j+k-1)}-p^{n(R_{j}+R'_{k}-j-k+1)} \right) \\
=&\sum_{n\geq 1}-\frac{u^{n}}{n} P_{R}(p^{n})
\end{align*}
where we've defined
\begin{align*}
P_{R}(x)& = \sum_{j,k\geq 1} (x^{R_{j}+R'_{k}-j-k+1} - x^{j+k-1})\\
& =\left( \sum_{j\geq 1} x^{R_{j}-j+\half}  \sum_{k\geq 1} x^{R_{k}-k+\half} \right)  -\frac{x}{(1-x)^{2}}\\
&=s_{\square}(x^{R+\rho})s_{\square}(x^{R'+\rho})  -\frac{x}{(1-x)^{2}}
\end{align*}
Since by equation \eqref{eqn: sR(ptorho)sR'(ptorho)= hook product}, we
have 
\[
s_{\square}\left(x^{R+\rho} \right) =
-s_{\square}\left(x^{-R'-\rho} \right),
\]
we see from the above equation that $P_{R}(x)$ is a sum of two
rational functions in $x$, each invariant under $x\leftrightarrow
x^{-1}$. Moreover, all but a finite number of terms in the sum
\[
\sum_{j,k\geq 1} \left(x^{R_{j}+R_{k}'-j-k+1} -x^{-j-i+1} \right) =P_{R}(x)
\]
cancel and so we deduce that $P_{R}(x)$ is a Laurent \emph{polynomial}
which is invariant under $x\leftrightarrow 1/x$. Consequently,
$P_{R}(x)$ is uniquely determined by its terms with positive
exponents. Since
\[
R_{j}+R'_{k}-j-k+1 = \begin{cases}
h_{jk}(R) & \text{if $(j,k)\in R$}\\
\text{negative}&\text{if $(j,k)\notin R$}
\end{cases}
\]
we see that the terms with positive exponent in the above expression for
$P_{R}$ are precisely $x^{h_{jk}(R)}$ where $(j,k)\in R$. Therefore
\[
P_{R}(x) = \sum_{j,k\in R}\left( x^{h_{jk}(R)}+ x^{-h_{jk}(R)} \right).
\]
Substituting back, we find
\begin{align*}
& \log \left(M(u,p) \prod_{j,k=1}^{\infty}
 \left(1-up^{R_{j}+R'_{k}-j-k+1} \right) \right) \\
=&\sum_{n\geq 1} -\frac{u^{n}}{n} \sum_{j,k\in R} \left( x^{h_{jk}(R)}+ x^{-h_{jk}(R)} \right)\\
=&\sum_{j,k\in R} \log\left(1-up^{h_{jk}(R)} \right)+ \log\left(1-up^{-h_{jk}(R)} \right)\\
=&\log \prod_{j,k\in R} \left(1-up^{h_{jk}(R)} \right) \left(1-up^{-h_{jk}(R)} \right)
\end{align*}
which proves equation~\eqref{eqn: 1st formula in hook lemma}. To prove
equation~\eqref{eqn: 2nd formula in hook lemma}, we observe that since
by equation~\eqref{eqn: sR(ptorho)sR'(ptorho)= hook product} 
\[
\sum_{k\geq 1}p^{R'_{k}-k+\half} = -\sum_{k\geq 1}p^{-R_{k}+k-\half} 
\]
so we have
\[
-\sum_{j,k\geq 1} \left(x^{-R_{j}+R_{k}+j-k} +x^{j+k-1} \right) = \sum_{j,k\in R}\left( x^{h_{jk}(R)}+ x^{-h_{jk}(R)} \right).
\]
Equation~\eqref{eqn: 2nd formula in hook lemma} then follows from a
similar logarithm argument as we did for equation~\eqref{eqn: 1st
formula in hook lemma}. 
\end{proof}

\subsection{The main derivation}

Recall that 
\[
Z'_{\ban} = M(p)^{-2} Z(\Fhat_{\ban})
\]
so that 
\[
Z'_{\ban} =
\sum_{R_{1}R_{2}R_{3}}(-Q_{1})^{R_{1}}(-Q_{2})^{R_{2}}(-Q_{3})^{R_{3}}\,
\Vtildesf_{R_{1}R_{2}R_{3}}
\,\Vtildesf_{R'_{1}R'_{2}R'_{3}} .
\]
We write
\[
Z'_{\ban }= \sum_{R} (-Q_{3})^{R}Z_{R}
\]
where
\begin{align*}
Z_{R}=&
\hspace{3.2cm}\sum_{R_{1}R_{2}} &&(-Q_{1})^{R_{1}}(-Q_{2})^{R_{2}}
\, \Vtildesf_{R_{1}R_{2}R}\, \Vtildesf_{R'_{1}R_{2}'R}
  \\
=&\, \, s_{R}(p^{-\rho})s_{R'}(p^{-\rho})
\sum_{A\, B\,  R_{1} R_{2}} &&(-Q_{1})^{R_{1}}(-Q_{2})^{R_{2}}
\cdot  s_{R'_{1}/A}(p^{-R-\rho})\cdot s_{R_{2}/A}(p^{-R'-\rho}) \\
&&&\hspace{3cm}\cdot s_{R_{1}/B}(p^{-R'-\rho})\cdot s_{R'_{2}/B}(p^{-R-\rho}) \\
=&\, \, s_{R}(p^{-\rho})s_{R'}(p^{-\rho})
\cdot \sum_{A\, B\, R_{2}}&&(-Q_{2})^{R_{2}}(-Q_{1})^{B}
\cdot s_{R_{2}/A}(p^{-R'-\rho})\cdot s_{R_{2}'/B}(p^{-R-\rho}) \\
&\hspace{3.2cm} \cdot \sum_{R_{1}}&& s_{R_{1}/B}(-Q_{1}p^{-R'-\rho} ) \cdot s_{R_{1}'/A}(p^{-R-\rho}) 
\end{align*}

From \cite[(2),page 93]{Macdonald} we have
\[
\sum_{R_{1}}s_{R_{1}/B}(x)s_{R_{1}'/A}(y) = \prod_{i,j=1}^{\infty}
(1+x_{i}y_{j}) \sum_{C} s_{A'/C}(x)s_{B'/C'}(y) .
\]

Using the above and equation~\eqref{eqn: sR(ptorho)sR'(ptorho)= hook
product} we get

\begin{align*}
Z_{R} = & (-1)^{R} \prod_{i,j\in R} (1-p^{h_{ij}(R)})^{-1}
(1-p^{-h_{ij}(R)})^{-1} \cdot \prod_{i,j=1}^{\infty}
(1-Q_{1}p^{i+j-1-R_{i}'-R_{j}}) \\
&\hspace{.5cm}\cdot \sum_{A\, B\, R_{2}\, C} (-Q_{2})^{R_{2}} (-Q_{1})^{B}
\cdot s_{R_{2}/A}(p^{-R'-\rho}) \cdot s_{R_{2}'/B}(p^{-R-\rho}) \\
& \hspace{4.8cm} \cdot s_{A'/C}(-Q_{1}p^{-R'-\rho} )\cdot s_{B'/C'}(p^{-R-\rho }). 
\end{align*}

Using equation~\eqref{eqn: 1st formula in hook lemma} from
Lemma~\ref{lem: infinite product = macmahon*finite hook product} we
get
\begin{align}\label{eqn: ZR=HR*three sums}
Z_{R} = H_{R} \sum_{C\, R_{2}} (-Q_{2})^{R_{2}} (-Q_{1})^{C}
&\cdot \sum_{A} s_{R_{2}/A}(p^{-R'-\rho})\cdot s_{A'/C}(-Q_{1}p^{-R'-\rho})\nonumber \\
&\cdot \sum_{B} s_{R_{2}'/B}(p^{-R-\rho})\cdot s_{B'/C'}(-Q_{1}p^{-R-\rho})
\end{align}
where
\[
H_{R} = (-1)^{R}M(Q_{1},p)^{-1} \prod_{i,j\in R}
\frac{(1-Q_{1}p^{h_{ij}(R)})(1-Q_{1}p^{-h_{ij}(R)})}{(1- p^{h_{ij}(R)})(1-
p^{-h_{ij}(R)})}
\]

Now using (see \cite[5.10 page 72]{Macdonald}
\[
\sum_{\nu} s_{\lambda /\nu}(x)s_{\nu /\mu}(y) = s_{\lambda /\mu}(x,y)
\]
and equation~\eqref{eqn: symmetric fnc involution identity }, we can
rewrite the second and third sums in equation~\ref{eqn: ZR=HR*three
sums} as 
\begin{align*}
\sum_{A} s_{R_{2}/A}(p^{-R'-\rho})\cdot s_{A'/C}(-Q_{1}p^{-R'-\rho})
&= s_{R_{2}/C'}(p^{-R'-\rho},Q_{1}p^{R+\rho}) \\
\sum_{B} s_{R_{2}'/B}(p^{-R-\rho})\cdot s_{B'/C'}(-Q_{1}p^{-R-\rho})
&= s_{R_{2}'/C}(p^{-R-\rho},Q_{1}p^{R'+\rho}) .
\end{align*}

Substituting back into equation~\eqref{eqn: ZR=HR*three sums} we get
\begin{align*}
Z_{R} = & H_{R} \sum_{C,R_{2}} (-Q_{2})^{R_{2}}(-Q_{1})^{C} \cdot 
s_{R_{2}/C'}(p^{-R'-\rho},Q_{1}p^{R+\rho}) \cdot s_{R'_{2}/C}(p^{-R-\rho
}Q_{1}p^{R'+\rho }) \\
=&  H_{R} \sum_{C,R_{2}} (Q_{1}Q_{2})^{R_{2}}\cdot s_{R_{2}/C'}(\mathbf{y},\mathbf{y'})\cdot s_{R_{2}'/C}(\mathbf{x},\mathbf{x}')
\end{align*}
where 
\[
\mathbf{y} = \{-Q_{1}^{-1} p^{-R'-\rho} \},\quad \mathbf{y'}=
\{-p^{R+\rho} \}, \quad \mathbf{x} = \{p^{-R-\rho} \}, \quad
\mathbf{x}' = \{Q_{1}p^{R'+\rho } \}. 
\]
We now use \cite[page 94, equation (b)]{Macdonald} to obtain
\begin{align*}
Z_{R}  = H_{R} &\prod_{i=1}^{\infty} (1-Q_{1}Q_{2})^{-1}\\
 \cdot & \prod_{j,k}(1+Q_{1}^{i}Q_{2}^{i} x_{j}y_{k})\cdot 
(1+Q_{1}^{i}Q_{2}^{i} x'_{j}y_{k})\cdot 
(1+Q_{1}^{i}Q_{2}^{i} x_{j}y'_{k})\cdot 
(1+Q_{1}^{i}Q_{2}^{i} x'_{j}y'_{k})
\end{align*}

We deal with each of four factors in the product over $j$ and $k$
using Lemma~\ref{lem: infinite product = macmahon*finite hook product}:

\begin{align*}
 \prod_{j,k}(1+Q_{1}^{i}Q_{2}^{i} x_{j}y_{k})&= \prod_{j,k}(1-Q_{1}^{i-1}Q_{2}^{i}\,p^{j+k-1-R_{k}'-R_{j} })\\
&= M(Q_{1}^{i-1}Q_{2}^{i},p)^{-1}\prod_{j,k\in R} \left(1-Q_{1}^{i-1}Q_{2}^{i} \,p^{h_{jk}(R)} \right)  \left(1-Q_{1}^{i-1}Q_{2}^{i} \,p^{-h_{jk}(R)} \right)  \\
 \prod_{j,k}(1+Q_{1}^{i}Q_{2}^{i} x'_{j}y_{k}) &= \prod_{j,k}(1-Q_{1}^{i}Q_{2}^{i}\,p^{-j+k-R_{k}'+R_{j}' })\\
&= M(Q_{1}^{i}Q_{2}^{i},p)\prod_{j,k\in R} \left(1-Q_{1}^{i}Q_{2}^{i} \,p^{h_{jk}(R)} \right)^{-1}  \left(1-Q_{1}^{i}Q_{2}^{i} \,p^{-h_{jk}(R)} \right)^{-1}  \\
 \prod_{j,k}(1+Q_{1}^{i}Q_{2}^{i} x_{j}y'_{k}) &= \prod_{j,k}(1-Q_{1}^{i}Q_{2}^{i}\,p^{j-k-1+R_{k}-R_{j} })\\
&= M(Q_{1}^{i}Q_{2}^{i},p) \prod_{j,k\in R} \left(1-Q_{1}^{i}Q_{2}^{i} \,p^{h_{jk}(R)} \right)^{-1}  \left(1-Q_{1}^{i}Q_{2}^{i} \,p^{-h_{jk}(R)} \right)^{-1}  \\
\prod_{j,k}(1+Q_{1}^{i}Q_{2}^{i} x'_{j}y'_{k}) &= \prod_{j,k}(1-Q_{1}^{i+1}Q_{2}^{i}\,p^{-j-k+1+R_{k}+R_{j}' })\\
&= M(Q_{1}^{i+1}Q_{2}^{i},p)^{-1}\prod_{j,k\in R} \left(1-Q_{1}^{i+1}Q_{2}^{i} \,p^{h_{jk}(R)} \right)  \left(1-Q_{1}^{i+1}Q_{2}^{i} \,p^{-h_{jk}(R)} \right) 
\end{align*}

Substituting back we see
\begin{align*}
&Z_{R} = H_{R} \cdot \prod_{i=1}^{\infty} \frac{M(Q_{1}^{i}Q_{2}^{i},p)^{2}}{ (1-Q_{1}^{i}Q_{2}^{i}) M(Q_{1}^{i-1}Q_{2}^{i},p)M(Q_{1}^{i+1}Q_{2}^{i},p)}\\
&\\
&\cdot \prod_{j,k\in R}
\frac{\left(1-Q_{1}^{i-1}Q_{2}^{i}p^{h_{jk}(R)}
\right)\left(1-Q_{1}^{i-1}Q_{2}^{i}p^{-h_{jk}(R)}
\right)\left(1-Q_{1}^{i+1}Q_{2}^{i}p^{h_{jk}(R)}
\right)\left(1-Q_{1}^{i+1}Q_{2}^{i}p^{-h_{jk}(R)}
\right)}{\left(1-Q_{1}^{i}Q_{2}^{i}p^{h_{jk}(R)}
\right)^{2}\left(1-Q_{1}^{i}Q_{2}^{i}p^{-h_{jk}(R)} \right)^{2}}
\end{align*}

Putting $H_{R}$ back in we get 
\begin{align*}
&Z_{R}= Z_{\inst} \cdot (-1)^{R} \cdot \prod_{j,k\in R} \frac{\left(1-Q_{1}p^{h_{jk}(R)} \right)\left(1-Q_{1}p^{-h_{jk}(R)} \right)}{\left(1-p^{h_{jk}(R)} \right)\left(1-p^{-h_{jk}(R)} \right)}\\
&\\
&\cdot  \prod_{i=1}^{\infty}
\frac{\left(1-Q_{1}^{i+1}Q_{2}^{i}p^{h_{jk}(R)}
\right)\left(1-Q_{1}^{i+1}Q_{2}^{i}p^{-h_{jk}(R)} \right)
\left(1-Q_{1}^{i-1}Q_{2}^{i}p^{h_{jk}(R)}
\right)\left(1-Q_{1}^{i-1}Q_{2}^{i}p^{-h_{jk}(R)} \right)}{\left(1-Q_{1}^{i}Q_{2}^{i}p^{h_{jk}(R)}
\right)^{2}\left(1-Q_{1}^{i}Q_{2}^{i}p^{-h_{jk}(R)} \right)^{2}}\\
&\\
&=  Z_{\inst} \cdot (-1)^{R} \cdot \prod_{j,k\in R}\\
&\cdot \prod_{i=1}^{\infty}
\frac{\left(1-Q_{1}^{i-1}Q_{2}^{i}p^{h_{jk}(R)}
\right)\left(1-Q_{1}^{i-1}Q_{2}^{i}p^{-h_{jk}(R)} \right)
\left(1-Q_{1}^{i-1}Q_{2}^{i}p^{h_{jk}(R)}
\right)\left(1-Q_{1}^{i-1}Q_{2}^{i}p^{-h_{jk}(R)}
\right)}{\left(1-Q_{1}^{i-1}Q_{2}^{i-1}p^{h_{jk}(R)}
\right)\left(1-Q_{1}^{i-1}Q_{2}^{i-1}p^{-h_{jk}(R)}
\right)\left(1-Q_{1}^{i}Q_{2}^{i}p^{h_{jk}(R)}
\right)\left(1-Q_{1}^{i}Q_{2}^{i}p^{-h_{jk}(R)} \right)}\\
\end{align*}

where we've defined
\begin{align*}
Z_{\inst} &=M(Q_{1},p)^{-1} \prod_{i=1}^{\infty}  \frac{M(Q_{1}^{i}Q_{2}^{i},p)^{2}}{ (1-Q_{1}^{i}Q_{2}^{i}) M(Q_{1}^{i-1}Q_{2}^{i},p)M(Q_{1}^{i+1}Q_{2}^{i},p)} \\
&= \prod_{i=1}^{\infty}
\frac{M(Q_{1}^{i}Q_{2}^{i},p)^{2}}{(1-Q_{1}^{i}Q_{2}^{i})M(Q_{1}^{i-1}Q_{2}^{i},p)M(Q_{1}^{i}Q_{2}^{i-1},p)}
\end{align*}

We now make the variable change given by equation~\eqref{eqn: variable
change Qs to q,y,etc}, sum over the remaining partition, and write the
result in terms of elliptic genera:
\begin{align*}
Z'_{\ban} =& Z_{\inst} \cdot \sum_{R} (y^{-1}Q)^{R} \\
& \cdot \prod_{i=1}^{\infty}\prod_{j,k\in R} \frac{(1-t^{h_{jk}
}yq^{i-1})(1-t^{-h_{jk} }yq^{i-1})(1-t^{h_{jk}
}y^{-1}q^{i})(1-t^{-h_{jk} }y^{-1}q^{i})}{(1-t^{h_{jk}
}q^{i-1})(1-t^{-h_{jk} }q^{i-1})(1-t^{h_{jk} }q^{i})(1-t^{-h_{jk}
}q^{i})}\\
&\\
=& Z_{\inst} \cdot \sum_{k=0}^{\infty} Q^{k} \Ell_{q,y}
((\CC^{2})^{[k]},t) \\
=& Z_{\inst}\cdot  \prod_{n=0}^{\infty}\prod_{m=1}^{\infty}\prod_{l,k\in
\ZZ} (1-t^{k}q^{n}y^{l}Q^{m})^{-c(nm,l,k)}\\
=& Z_{\inst}\cdot  \prod_{n=0}^{\infty}\prod_{m=1}^{\infty}\prod_{l,k\in
\ZZ} (1-t^{k}q^{n}y^{l}Q^{m})^{-c(4nm-l^{2},k)}
\end{align*}
where the last three equalities come from equation~\eqref{eqn: formula
for Ell via Atiyah-Bott localization}, Theorem~\ref{thm: DMVV
formula}, and Proposition~\ref{prop: generating function for c(a,k) =
c(4n-l^2,k) =c(n,l,k)} respectively.

Returning to the DT variables via
\[
t=p, \quad q=Q_{1}Q_{2}, \quad y=Q_{1}, \quad Q=Q_{1}Q_{3},
\]
reindexing by
\[
d_{1} = n+l+m,\quad d_{2}=n,\quad d_{3}=m,
\]
and observing that 
\[
||d|| =
2d_{1}d_{2}+2d_{2}d_{3}+2d_{3}d_{1}-d_{1}^{2}-d_{2}^{2}-d_{3}^{2} =4mn-l^{2}
\]
we find
\[
Z'_{\ban} = Z_{\inst} \cdot \prod_{d_{3}=1}^{\infty}\prod_{
d_{2}=0}^{\infty}\prod_{d_{1},k\in \ZZ}
(1-p^{k}Q_{1}^{d_{1}}Q_{2}^{d_{2}}Q_{3}^{d_{3}})^{-c(||d||,k)}
\]
Observing further that when $d_{1}<0$ and $d_{3}>0$, $||d|| =
4d_{3}d_{1}-(d_{2}-d_{1}-d_{3})^{2}<-1$ and so $c(||d||,k)=0$, we get 
\begin{align*}
Z(\Fhat_{\ban} )&= M(p)^{2}\cdot  Z'_{\ban}\\
&  = M(p)^{2}\cdot Z_{\inst}\cdot 
\prod_{d_{3}=1}^{\infty} \prod_{d_{2},d_{3}=0}^{\infty} \prod_{k\in
\ZZ} (1-p^{k}Q_{1}^{d_{1}}Q_{2}^{d_{2}}Q_{3}^{d_{3}})^{-c(||d||,k)}
\end{align*}

Finally we claim that
\[
M(p)^{2}\cdot Z_{\inst} = \prod_{(*)}
(1-p^{k}Q_{1}^{d_{1}}Q_{2}^{d_{2}}Q_{3}^{d_{3}})^{-c(||d||,k)}
\]
where the product is over
\[
(*)\quad \quad  \quad d_{3}=0,\quad d_{1},d_{2}\geq 0,\quad \text{and } \begin{cases}
k\in \ZZ & (d_{1},d_{2})\neq (0,0),\\
k>0 & (d_{1},d_{2}) = (0,0).
\end{cases} 
\]
Indeed, if $d_{3}=0$, then $||d||=-(d_{2}-d_{1})^{2}$ and by
Corollary~\ref{cor: c(-1,k) and c(0,k)} the product over $(*)$ reduces
to the terms where
\[
k\in \ZZ ,\quad (d_{1},d_{2}) = (d,d-1),(d-1,d),(d,d),\quad  d>1
\]
and the special case $d_{1}=d_{2}=0, k>0$. Applying
Corollary~\ref{cor: c(-1,k) and c(0,k)}  we easily deduce the above
claim. 

Substituting the claimed equation back into the previous equation for
$Z(\Fhat_{\ban})$ then finishes the proof of Theorem~\ref{thm: formula
for Zban}.  \qed 
    
\section{Geometry of the Banana manifold}\label{sec: geom of ban man}

In this section we compute the Hodge numbers of the banana manifold
$X_{\ban }$, we show that the fiber classes are spanned by the banana
curves $C_{1},C_{2},C_{3}$, and we prove Proposition~\ref{prop: toric
description of FhatBan} which describes the formal neighborhood of the
singular fibers.

\newcommand{\othermark}[1]{{#1}^{\scriptscriptstyle{\#}}}

Let $p:S\to \PP^{1}$ be a rational elliptic surface with 12 singular
fibers, each having one node. Let $\othermark{p}:\othermark{S}\to
\PP^{1}$ be an isomorphic copy of $S$ and consider the fibered product
\[
X_{\sing} = S\times_{\PP^{1}}\othermark{S}.
\]
$X_{\sing}$ is singular at the 12 points where both $p$ and
$\othermark{p}$ are not smooth, namely at the product of the nodes. 
To see that these points are conifold singularities, note that for
each node $n\in S$, and corresponding node $\othermark{n}\in \othermark{S}$, there
exists formal local coordinates
$(x ,y ),(\othermark{x} ,\othermark{y} ),t $ about $n ,$
$\othermark{n} $, and $p(n )=\othermark{p}(\othermark{n} )\in
\PP^{1}$ respectively such that the maps $p$ and $\othermark{p}$ are
given by $x y =t $ and
$\othermark{x} \othermark{y} =t $. Consequently,
$x y =\othermark{x} \othermark{y} $ is the local equation
of $X_{\sing}$ at $(n ,\othermark{n} )$. 

Let $\Delta \subset X_{\sing}$ be the divisor given by the diagonal in
the fibered product. From the above local description of the
singularities, all of which lie on $\Delta$, we see that 
\[
X_{\ban} =\Bl_{\Delta}(X_{\sing}),
\]
the blowup of $X_{\sing}$ along the diagonal is smooth and 
\[
X_{\ban}\to X_{\sing}
\]
is a conifold resolution.

\subsection{Relation to the Schoen threefold and the Hodge numbers}

\begin{lemma}\label{lem: e(Xban)=24}
$e(X_{\ban} )=24$. 
\end{lemma}
\begin{proof}
Since the map $X_{\ban }\to X_{\sing}$ contracts 12 $\PP^{1}$s to 12
singular points, we can write the following equation in the
Grothendieck group of varieties,
\[
[X_{\ban}] = [X^{o}_{\sing}] + 12[\PP^{1}]
\]
where $X_{\sing}^{o}\subset X_{\sing}$ is the non-singular
locus. Since $e(-)$ is a homomorphism from the Grothendieck group to
the integers, we see that $e(X_{\ban})=e(X_{\sing}^{o})+24$. We claim
the Euler characteristic of the fibers of $X_{\sing}^{o}\to \PP^{1}$
are all zero. The smooth fibers are $E\times E$, a product of smooth
elliptic curves and hence have zero Euler characterisitic. The
singular fibers have a $\CC^{*}\times \CC^{*}$ action constructed in
\S~\ref{subsec: Mordell weil gps and actions}. The fixed points of the
action on the fibers of $X_{\sing}\to \PP^{1}$ are exactly the
conifold points and so the action on the fibers of $X^{o}_{\sing}\to
\PP^{1}$ is free. Consequently, the Euler characteristics of these
fibers are zero. The lemma follows.
\end{proof}

The banana manifold is related to the Schoen Calabi-Yau threefold by a
conifold transition. This will allow us to compute the Hodge numbers
of the banana manifold in terms of the (well-known) Hodge numbers of
the Schoen threefold. The Schoen threefold can be defined as
\[
X_{\Sch} \subset \PP^{2}\times \PP^{2}\times \PP^{1},
\]
the intersection of two generic hypersurfaces of multi-degree
$(3,0,1)$ and $(0,3,1)$. It is a simply connected Calabi-Yau threefold
with $h^{1,1}(X_{\Sch})=19$ and  $h^{2,1}(X_{\Sch})=19$.

\begin{proposition}\label{prop: h11(Xban)=20 h21(Xban)=8}
$X_{\ban}$ is a simply connected Calabi-Yau with $h^{2,1}(X_{\ban})=8$
 and $h^{1,1}(X_{\ban})=20$. 
\end{proposition}
\begin{proof}
We first show that there is a conifold transition from $X_{\Sch}$ to
$X_{\ban}$. The generic hypersurface in $\PP^{2}\times \PP^{1}$ of
degree $(3,1)$ is a rational elliptic surface $S\to \PP^{1}$ and
consequently the projections of $X_{\Sch} \subset \PP^{2}\times
\PP^{2}\times \PP^{1}$ onto various subfactors realize $X_{\Sch}$ as a
fibered product
\[
X_{\Sch} = S_{0}\times_{\PP^{1}}S_{1}
\]
where $S_{i}\to  \PP^{1}$ are distinct generic rational surfaces. We
may choose a 1-parameter family of elliptic surfaces $S_{t}$ which interpolates
between $S_{0}$ and $S_{1}$. Then the family of threefolds
\[
X_{t} = S_{0}\times_{\PP^{1}} S_{t}
\]
are all Schoen threefolds for $t\neq 0$ and $X_{0}=X_{\sing}$ and
since $X_{\ban}\to X_{\sing}$ is a conifold resolution, we get a
conifold transition $X_{\Sch} \rightsquigarrow X_{\ban}$. Since
conifold transitions preserve simple connectivity and the Calabi-Yau
condition, we see that $X_{\ban}$ is a simply connected Calabi-Yau
threefold. The change in $h^{1,1}$ through a conifold transition is
determined by the codimension of the family of singular threefolds
inside the full deformation space of the threefold (see
\cite[\S~3.1]{Morrison-LookingGlass}).  The family of Schoen
threefolds is 19 dimensional: there are the two 8 dimensional families
of rational elliptic surfaces $S_{i}\to \PP^{1}$ and the three
dimensional family of isomorphisms between the bases of $S_{i}\to
\PP^{1}$ required to form the fibered product $S_{0}\times_{\PP^{1}}
S_{1}$. The locus of such fibered products with
\[
\delta =12
\]
conifold singularities is 8 dimensional: the 12 nodes of $S_{0}$ and
$S_{1}$ must occur over the same fibers which implies that $S_{0}\cong
S_{1}$ and the isomorphism of the base is the identity. The
codimension $\sigma $ is thus
\[
\sigma =19-8=11.
\]
Following
\cite[\S~3.1]{Morrison-LookingGlass}, we compute:
\begin{align*}
h^{1,1}(X_{\ban}) &= h^{1,1}(X_{\Sch}) + \delta -\sigma \\
&= 19+12-11=20.
\end{align*}
Then since $e(X_{\ban})=24 = 2(h^{1,1}(X_{\ban})-h^{2,1}(X_{\ban}))$
we see that $h^{2,1}(X_{\ban})=8$ (in particular, the 8 dimensional
space of banana manifolds constructed above is the whole deformation
space). 
\end{proof}

\begin{lemma}\label{lem: C1,C2,C3 span fiber classes}
Let $C_{1}$, $C_{2}$, $C_{3}$ be the banana curves in a singular fiber
$F_{\sing}$. Then $\Gamma =\Ker (\pi_{*}:H_{2}(X_{\ban},\ZZ )\to
H_{2}(\PP^{1},\ZZ ))$ is spanned by the classes of $C_{1}$, $C_{2}$,
and $C_{3}$. 
\end{lemma}

\begin{proof}
By the previous discussion of the conifold transition, we have that 
\[
H_{2}(X_{\ban},\ZZ )\cong H_{2}(X_{\Sch},\ZZ )\oplus \ZZ 
\]
where the $\ZZ$ factor is spanned by the exceptional curves of the
conifold resolution, in particular, the exceptional curves are all
homologous. Let 
\[
p_{1},p_{2}:X_{\ban} \to S
\]
be the projections on to the first and second factors of the fibered
product. Then 
\[
\Gamma = \Ker (p_{1})_{*} \cup  \Ker (p_{2})_{*} .
\]
This follows since if a connected algebraic 1-cycle in $X_{\ban}$ maps to a
point in $\PP^{1}$, then by properties of the fiber product, it either
maps to a point under $p_{1}$ or $p_{2}$. Moreover, it is enough to
consider algebraic cycles since $H_{2}(X_{\ban},\ZZ )\cong
H^{4}(X_{\ban},\ZZ )\subset H^{2,2}(X_{\ban})  $ has no torsion and the
Hodge conjecture holds for threefolds.

Let $x_{1},\dotsc ,x_{12}\in \PP^{1}$ be points corresponding to
singular fibers and we label the banana curves
$C_{1}(i),C_{2}(i),C_{3}(i)$ in the singular fiber over $x_{i}$ such
that $C_{3}(i)$ is an exceptional curve for the conifold resolution
and $C_{1}(i)$ and  $C_{2}(i)$ are such that
\[
p_{j}^{-1}(p_{j}(C_{3}(i))) = C_{j}(i)\cup C_{3}(i), \quad j=1,2.
\]
The fibers of $p_{j}:X_{\ban}\to S$ are all irreducible except
for the fibers $C_{j}(i)\cup C_{3}(i)$ for $i=1,\dotsc ,12$. Since the
exceptional curves are all homologous, we have a single class
$C_{3}=C_{3}(i)$ for all $i$ and since the fibers of $p_{j}$ are all
homologous we get
\[
C_{1}(i)+C_{3} = C_{1}(i')+C_{3},\text{ and }   C_{2}(i)+C_{3} =
C_{2}(i')+C_{3} .
\]
Hence we all $C_{1}(i)$ are homologous to a single class $C_{1}$ and
all $C_{2}(i)$ are homologous to a single class $C_{2}$ and $\Gamma$
is spanned by $C_{1},C_{2},C_{3}$.
   
\end{proof}

\subsection{Proof of Proposition~\ref{prop: toric description of
FhatBan}.}\label{subsec: proof of prop about FhatBan}

Let $C_{\sing}\subset S$ be a singular fiber of $S\to \PP^{1}$ and let
$\othermark{C}_{\sing}\subset \othermark{S}$ be an isomorphic copy so
that 
\[
X_{\sing}=S\times_{\PP^{1}}\othermark{S}\text{ and
}F_{\sing}=C_{\sing}\times \othermark{C}_{\sing}.
\]
$C_{\sing}$ is a nodal rational curve whose normalization $\PP^{1}\to
C_{\sing}$ identifies the points $0,\infty \in \PP^{1}$ to the nodal
point $n\in C_{\sing}$. Thus $F^{\norm}_{\sing}$ is isomorphic to
$\PP^{1}\times \PP^{1}$ and the map $F^{\norm}_{\sing}\to F_{\sing}$
identifies $0\times \PP^{1}$ and $\infty \times \PP^{1}$ to $n\times
\othermark{C}_{\sing}$ and identifies $\PP^{1}\times 0$ and
$\PP^{1}\times \infty$ to $C_{\sing}\times \othermark{n}$. Note that
this map is $\CC^{*}\times \CC^{*}$ equivariant with respect to the
usual toric action on $\PP^{1}\times \PP^{1}$ and our constructed
action on $F_{\sing}$.

The formal neighborhood $\Fhat_{\ban}$ is obtained from $\Fhat_{\sing}$
by blowing up the diagonal $\Delta \subset \Fhat_{\sing}$. The blow
down $\Fhat_{\ban}\to \Fhat_{\sing}$ contracts the exceptional
$\PP^{1}$ to the conifold point $n\times \othermark{n}$ and is an
isomorphism elsewhere. We study this blowup and the normalizations in
formal local coordinates about $n\times \othermark{n}$:

Let $(x,y)$ and $(\othermark{x},\othermark{y})$ be formal local
coordinates on $S$ and $\othermark{S}$ about the points $n$ and
$\othermark{n}$ such that the maps $S\to \PP^{1}$ and
$\othermark{S}\to \PP^{1}$ are given by $t=xy$ and
$t=\othermark{x}\othermark{y}$ where $t$ is a formal local coordinate
on $\PP^{1}$. Moreover, we may choose the coordinates so that the
action of $(\lambda ,\othermark{\lambda})\in \CC^{*}\times \CC^{*}$ on
$\Fhat_{\sing}$ is given by
\[
(x,y,\othermark{x},\othermark{y})\mapsto (\lambda
x,\lambda^{-1}y,\othermark{\lambda}\othermark{x},(\othermark{\lambda})^{-1}\othermark{y}).
\]
Then the formal neighborhood of $n\times \othermark{n}\in
\Fhat_{sing}$ is isomorphic to 
\[
\left\{xy=\othermark{x}\othermark{y} \right\} \subset \Spec \CC
[[x,y,\othermark{x},\othermark{y}]] 
\]
where the closed fiber $F_{\sing}\subset \Fhat_{\sing}$ is given by
\[
\left\{xy=\othermark{x}\othermark{y}=0 \right\} .
\]
The blowup of $\Fhat_{\sing}$ along the diagonal $\Delta
=\{x=\othermark{x},y=\othermark{y} \}$ is {canonically} isomorphic to
the blowup along any of the planes\footnote{There are canonically two
small resolutions of the conifold singularity
$xy=\othermark{x}\othermark{y}$. These two are given by blowing up any
plane given by affine cone over a line in one of the two rulings of
the quadric surface $\{xy=\othermark{x}\othermark{y} \}\subset
\PP^{3}$. }
\[
\left\{ax=\othermark{a}\othermark{x},\othermark{a}y
=a\othermark{y}\right\}, \quad \quad (a:\othermark{a})\in \PP^{1}.
\]
Choosing $(a:\othermark{a})=(1:0)$, we get two affine toric charts for
the blow up with coordinate rings given by
\begin{align*}
\CC
[[x,\othermark{x},y,\othermark{y}]][u]/(x-u\othermark{y},\othermark{x}-uy)&\cong
\CC [[y,\othermark{y}]][u],\\
\CC
[[x,\othermark{x},y,\othermark{y}]][v]/(y-v\othermark{x},\othermark{y}-vx)&\cong
\CC [[x,\othermark{x}]][v].
\end{align*}
The coordinate change between the charts is given by
\[
u=v^{-1},\quad y=v\othermark{x}, \quad \othermark{y}=vx
\]
and the induced $(\lambda ,\othermark{\lambda})$ action is given by
\[
(\lambda x,\othermark{\lambda}\othermark{x},(\lambda
\othermark{\lambda})^{-1}v),\quad
(\lambda^{-1}y,(\othermark{\lambda})^{-1}\othermark{y},\lambda
\othermark{\lambda}u).
\]

The coordinates on the blowup of the formal neighborhood of $n\times
\othermark{n}\in \Fhat_{\sing}$ and the corresponding blowdown are 
encoded in the following momentum ``polytopes'' where the coordinate
lines are labelled by their corresponding variables:

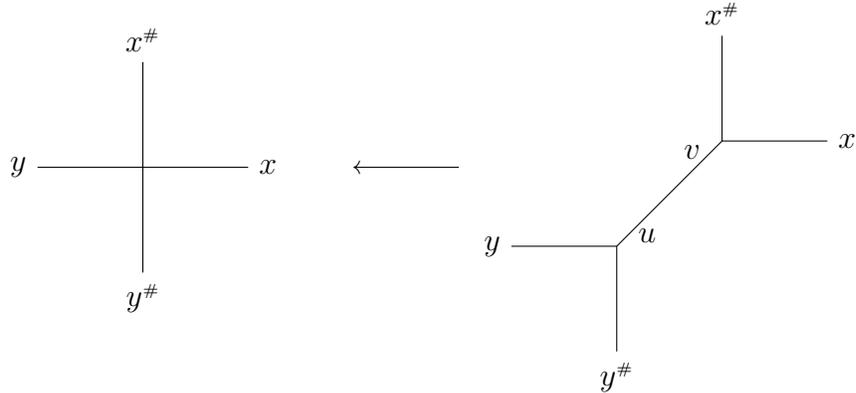
\begin{tikzpicture}[scale=0.7]\label{Fig: momentum polytopes of Fsing
and Fban}
\draw (16,8)--(14,8)--(14,10);
\node [right] at(16,8) {$x$};
\node [above] at (14,10 ) {$\othermark{x}$};
\node [left] at (13.8,7.8) {$v$};
\draw (14,8)--(12,6)--(12,4);
\draw (12,6)--(10,6);
\node [below] at(12,4) {$\othermark{y}$};
\node [left] at (10,6) {$y$};
\node [right] at (12.2,6.2) {$u$};
\draw [->](9,7.5)--(7,7.5);
\draw (5,7.5)--(1,7.5);
\node [right] at(5,7.5) {$x$};
\node [left] at(1,7.5) {$y$};
\draw (3,5.5)--(3,9.5);
\node [above] at(3 ,9.5 ) {$\othermark{x}$};
\node [below] at(3 ,5.5 ) {$\othermark{y}$};
\end{tikzpicture}
\vspace{1cm}

The $x$ and $y$ coordinates about $n\in S$ (which correspond to the
two branches of the node in $C_{\sing}$) become local coordinates near
$0$ and $\infty$ in $\PP^{1}$, the normalization of $C_{\sing}$. 
  
Thus $F^{\norm}_{\sing}$ and $F^{\norm}_{\ban}$, when endowed with
local coordinate rings given by $\sigma^{*}\Ohat_{\Fhat_{\sing}}$ and
$\tau^{*}\Ohat_{\Fhat_{\ban}}$ respectively, have momentum polytopes
given by:
\bigskip

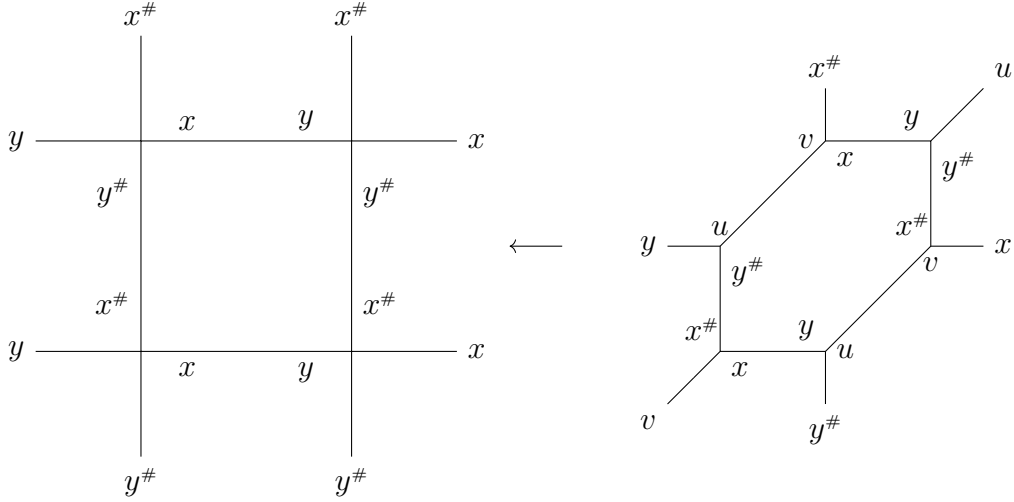
\begin{tikzpicture}[scale=0.7]
\draw (2,0)--(2,2)--(0,2);
\node [below right] at(2.5,2) {$x$};
\node [above left] at(2,2.5) {$\othermark{x}$};
\node [below] at(2,0) {$\othermark{y}$};
\node [left] at(0,2) {$y$};
\draw (2,2)--(2,6)--(6,6)--(6,2)--(2,2);
\draw (8,6)--(6,6)--(6,8);
\node [right] at(8,6) {$x$};
\node [above] at(6,8) {$\othermark{x}$};
\node [below right] at(6,5.5) {$\othermark{y}$};
\node [above left] at(5.5,6) {$y$};
\draw (0,6)--(2,6)--(2,8);
\node [left] at(0,6) {$y$};
\node [above] at(2,8) {$\othermark{x}$};
\node [below left] at(2,5.5) {$\othermark{y}$};
\node [above right] at(2.5,6) {$x$};
\draw (6,0)--(6,2)--(8,2);
\node [below] at(6,0) {$\othermark{y}$};
\node [right] at(8,2) {$x$};
\node [above right] at(6,2.5) {$\othermark{x}$};
\node [below left] at(5.5,2) {$y$};
\draw [<-](9,4)--(10,4);
\draw (13,2)--(13,4)--(15,6)--(17,6)--(17,4)--(15,2)--(13,2);
\draw (13,2)--(12.5,1.5);
\draw (12.5,1.5)--(12,1);
\node [below left] at (12,1) {$v$};
\node [below right] at(13,2) {$x$};
\node [above left] at(13.2,2) {$\othermark{x}$};
\draw (13,4)--(12,4);
\node [left] at(12,4) {$y$};
\node [above] at(13,4) {$u$};
\node [below right] at(13,4) {$\othermark{y}$};
\draw (15,6)--(15,7);
\node [above] at(15,7) {$\othermark{x}$};
\node [below right] at (15,6) {$x$};
\node [left] at(15,6) {$v$};
\draw (17,6)--(17.5,6.5);
\draw (17.5,6.5)--(18,7);
\node [above right] at(18,7) {$u$};
\node [below right] at(17,6) {$\othermark{y}$};
\node [above left] at(17,6) {$y$};
\draw (17,4)--(18,4);
\node [right] at(18,4) {$x$};
\node [above left] at(17.2,4) {$\othermark{x}$};
\node [below] at(17,4) {$v$};
\draw (15,2)--(15,1);
\node [below] at(15,1) {$\othermark{y}$};
\node [right] at(15,2) {$u$};
\node [above left] at(15,2) {$y$};
\end{tikzpicture}

\bigskip

We see that 
\[
F^{\norm}_{\ban}\cong \Bl_{(0,0)(\infty ,\infty
)}(\PP^{1}\times \PP^{1}),\quad \quad \Fhat^{\norm}_{\ban}\cong
\widehat{\Bl( \PP^{1}\times \PP^{1})}
\]
as asserted.

\section{BPS invariants from a Donaldson-Thomas partition function in
product form.}\label{sec: BPS invariants from product DT partition function}

The Gopakumar-Vafa invariants (a.k.a. BPS invariants) $n_{\beta
}^{g}(X)$ can be defined in terms of Gromov-Witten or
Donaldson-Thomas invariants and (conjecturally) have better finiteness
properties. Unlike the Donaldson-Thomas invariants, it is expected
that there are only a finite number of non-zero Gopakumar-Vafa
invariants for each curve class $\beta$.

One definition of Gopakumar-Vafa invariants, which is equivalent
to the usual one given in terms of Gromov-Witten invariants via the
Gopakumar-Vafa formula, is the following:
\begin{def-theorem}\label{def-thm: GV invariants in terms of DT
partition fnc}
Let $X$ be any Calabi-Yau threefold and suppose that the 
Donaldson-Thomas partition function  is given by
\begin{equation}\label{eqn: product formulation of Z(X)}
Z^{\DTss}(X) = \prod_{ \beta \in H_{2}(X) } \prod_{k\in \ZZ} \left(1-p^{k}Q^{\beta} \right)^{-a(\beta ,k)}
\end{equation}
for some integers $a(\beta ,k)\in \ZZ$. Then for $\beta\neq 0 \in H_2(X)$
and $g\in \ZZ$ the Gopakumar-Vafa invariants $n^{g}_{\beta}(X)$ can be
defined by the formula
\[
\sum_{g} n^{g}_{\beta}(X)\, \left(y^{\half} + y^{-\half} \right)^{2g}
=  \left(y^{\half} + y^{-\half} \right)^{2} \,\sum_{k\in \ZZ} a(\beta
,k) \, (-y)^{k}. 
\]
This is equivalent, assuming the Gromov-Witten/Donaldson-Thomas
correspondence, to defining $n^{g}_{\beta}(X)$ in terms of the
Gromov-Witten invariants via the Gopakumar-Vafa formula.
\end{def-theorem}

\begin{remark}
It is expected that $n^{g}_{\beta}(X)$ is zero unless $0\leq g\leq
g_{\mathsf{max}}$ where $g_{\mathsf{max}}$ is the maximal arithmetic
genus of curves in the class $\beta$. Note that the left hand side of
the above formula is a palandromic polynomial in $y$ (i.e. invariant
under $y \leftrightarrow y^{-1}$). One can see that the invariants
$n^{g}_{\beta}(X)$ are well-defined by the above definition as
follows. Any power series $Z\in \ZZ ((p))[[Q]]$ can be uniquely
written in the form of equation~\eqref{eqn: product formulation of
Z(X)} for some collection of integers $a(\beta ,k)$. By a theorem of
Bridgeland \cite{Bridgeland-PTDT}, the right hand side of
equation~\eqref{eqn: product formulation of Z(X)} is a rational
function of $y$, invariant under $y\leftrightarrow y^{-1}$, possibly
having poles at $y=-1$. Such functions have a basis given by the
functions
\[
\left(y^{\half}+y^{-\half} \right)^{2g} = y^{-g}(1+y)^{2g}, \quad g\in \ZZ 
\]
and so the left hand side of \eqref{eqn: product formulation of Z(X)}
is well defined and uniquely determines $n^{g}_{\beta}(X)$. The
conjecture that $n_{\beta}^{g}(X)=0$ if $g<0$ is equivalent to the
right hand side of \eqref{eqn: product formulation of Z(X)} not having
a pole.
\end{remark}

\begin{proof}
The Gopakumar-Vafa formula \cite{Go-Va} expresses the reduced
Gromov-Witten potential function $F'(X)$ of a Calabi-Yau threefold
$X$ in terms of conjecturally integer invariants $n_{\beta}^{g}(X)$,
which are commonly called Gopakumar-Vafa invariants, or BPS
invariants:

\[
F'(X) =\sum_{g\geq 0}\sum_{\beta \neq 0} GW^{g}_{\beta}(X)
\lambda^{2g-2} Q^{\beta} = \sum_{g\geq 0}\sum_{\beta \neq 0}\sum_{m>0} n_{\beta}^{g}(X)
\frac{1}{m}\left(2\sin \frac{m\lambda}{2} \right)^{2g-2} Q^{m\beta }.
\]

One can re-express the above formula in terms of the partition
function $Z'(X) = \exp\left(F'(X) \right)$ as follows\footnote{This expression essentially appears
in \cite[eqn~18]{Katz-Snowbird}, although our convention for $\binom{n}{k}$ allows us
to write the formula more uniformly.} 
:
\[
Z'(X) = \prod_{\beta \neq 0} \prod_{m\in \ZZ} \left(1-Q^{\beta}p^{m}
\right)^{-\sum_{g\geq 0} n^{g}_{\beta}(X) \binom{2g-2}{g-1-m}(-1)^{m}}
\]
where $p=\exp\left(i\lambda \right)$ and
$\binom{n}{k}=\frac{n(n-1)\dotsb (n-k+1)}{k!}$ is defined for all
$k\geq 0$ and $n\in \ZZ$.

The Donaldson-Thomas partition function can always be written in the
following form
\[
Z^{\DTss}(X) = \prod_{\beta}\prod_{m\in \ZZ} \left(1-Q^{\beta}p^{m} \right)^{-c(\beta ,m)}
\]
for some $c(\beta ,m)\in \ZZ $. Then for $\beta \neq 0$, 
\[
c(\beta ,m) = \sum_{n=0}^{\beta}n_{\beta}^{g}(X) \binom{2g-2}{g-1-m}(-1)^{m}.
\]  
The Definition-Theorem is then  an easy consequence of the binomial theorem.

\end{proof}

\bigskip


\appendix
\section{The Gromov-Witten potentials are Siegel modular
forms (with Stephen Pietromonaco)}\label{Appendix: GW potentials are Siegel modular forms}

\subsection{Overview.}
Let $F_{g}(Q_{1},Q_{2},Q_{3})$ be the genus $g\geq 2$ Gromov-Witten
potential for fiber classes in $X_{\ban}$. Namely, let
\[
F_{g}(Q_{1},Q_{2},Q_{3})  = \sum_{d_{1},d_{2},d_{3}\geq 0}
GW^{g}_{\dvec } (X_{\ban})\, Q_{1}^{d_{1}} Q_{2}^{d_{2}} Q_{3}^{d_{3}}
\]
where $GW^{g}_{\dvec} (X_{\ban})$ is the genus $g$
Gromov-Witten invariant of $X_{\ban}$ in the class
$\beta_{\dvec}  = d_{1}C_{1}+d_{2}C_{2}+d_{3}C_{3}$.

Assuming that the Gromov-Witten/Donaldson-Thomas correspondence
conjectured in \cite{MNOP1} holds, we may compute
$F_{g}(Q_{1},Q_{2},Q_{3})$ using the formula for the Donaldson-Thomas
partition function derived in the main text. The main result of this
appendix is that $F_{g}$ is an explicit genus 2 Siegel modular form of
weight $2g-2$.

\begin{definition}\label{defn: genus 2 Siegel modular form}
Let $\Omega  = \left(\begin{smallmatrix} \tau &z\\
z&\sigma \end{smallmatrix} \right)\in \mathbb{H}_{2}$ be the standard
coordinates on the genus 2 Siegel upper half plane. A holomorphic
(resp. meromorphic) genus 2 Siegel modular form of weight $k$ is a
holomorphic (resp. meromorphic) function $F(\Omega )$ satisfying
\[
F\left((A\Omega +B)(C\Omega +D)^{-1} \right) =
\operatorname{det}\left( (C\Omega +D)^{k} \right)F(\Omega )
\]
for all $\left(\begin{smallmatrix} A&B\\ C&D \end{smallmatrix}
\right)\in Sp_{4}(\ZZ )$ (c.f. \cite{Eichler-Zagier}). We denote the
space of meromorphic genus 2 Siegel forms of weight $k$ by
$\mathsf{Siegel}_{k}$. 
\end{definition}

As we will detail in \S~\ref{subsec: modular forms and lifts}, a
standard way to construct a genus 2 Siegel modular form of weight
$2g-2$ is to take the so-called Maass lift of a Jacobi form of weight
$2g-2$ and index 1. Moreover, such Jacobi forms are easily obtained by
taking a modular form of weight $2g$ and multiplying it by
$\phi_{-2,1}$, the unique weak Jacobi form of weight $-2$ and index
1. Some authors call this the Skoruppa lift of the modular
form. Schematically we have

\[
\begin{tikzcd}[column sep = large]
\Mod_{2g} \arrow[r,"\text{Skoruppa}"]&
\Jac_{2g-2,1} \arrow[r,"{\text{Maass}}"] &
\Siegel_{2g-2}
\end{tikzcd}
\] 
We call the composition the Skoruppa-Maass lift. It takes weight $2g$
modular forms to genus 2, weight $2g-2$, Siegel modular
forms\footnote{This lift is different from the famous Saito-Kurokawa
lift which also takes $\Mod_{2g}$ to $\Siegel_{2g-2}$. While both use
the Maass lift, the Saito-Kurokawa lift uses a combination of the
Shimura correspondence and a lift studied by Eichler-Zagier to go from
$\Mod_{2g}$ to $\Jac_{2g-2,1}$ \cite{Eichler-Zagier}.}. Our main
result is:
\begin{theorem}\label{thm: Fg are Siegel mod forms}
Assume that the Gromov-Witten/Donaldson-Thomas correspondence holds
for $X_{\ban}$. Then for $g\geq 2$, the genus $g$ Gromov-Witten
potentials $F_{g}(Q_{1},Q_{2},Q_{3})$ of $X_{\ban}$ are meromorphic
genus 2 Siegel modular forms of weight $2g-2$ where
\[
Q_{1} = e^{2\pi i\,z},\quad Q_{2} = e^{2\pi i(\tau -z)},\quad Q_{3} =
e^{2\pi i(\sigma -z)} .
\]
Specifically, $F_{g}$ is the Skoruppa-Maass lift of $a_{g}E_{2g}(\tau
)$, the $2g$-th Eisenstein series times the constant
$a_{g}=\frac{6|B_{2g}|}{g(2g-2)!}$ where $B_{2g}$ is the
$2g$-th Bernoulli number. 
\end{theorem}

The ring of holomorphic, even weight, genus 2 Siegel modular forms is
a polynomial ring generated by the Igusa cusp forms $\chi_{10}$ and
$\chi_{12}$ of weight 10 and 12, and the Siegel Eisenstein series
$\mathcal{E}_{4}$ and $\mathcal{E}_{6}$ of weight 4 and 6
\cite{Igusa-1962,Igusa-1967}. Although $F_{g}$ are meromorphic, the
denominators can be determined explicitly:

\begin{corollary}\label{cor: chi10^{g-1}Fg is holomorphic} For $g\geq
2$, the product $\chi_{10}^{g-1}\cdot F_{g}$ is a holomorphic Siegel
form of weight $12g-12$.
\end{corollary}
This corollary follows from Aoki's proof of \cite[Thm~14]{Aoki}.  In
terms of the generators $\chi_{10},\chi_{12},\mathcal{E}_{4},
\mathcal{E}_{6}$, the first few potentials are given explicitly as
\begin{align*}
F_{2}& = \frac{1}{240}\left(\frac{\chi_{12}}{\chi_{10}} \right), \\
F_{3}&= \frac{-1}{60480}\left(6\,\mathcal{E}_{4}-5\left(\frac{\chi_{12}}{\chi_{10}}
\right)^{2} \right),\\
F_{4}&=
\frac{1}{3628800}\left(\frac{35}{2}\left(\frac{\chi_{12}}{\chi_{10}}
\right)^{3}-\frac{63}{2}\left(\frac{\chi_{12}}{\chi_{10}}
\right)\mathcal{E}_{4}+ 15 \,\mathcal{E}_{6} \right),\\
F_{5}&=\frac{-1}{106444800}\left(\frac{-175}{3}\left(\frac{\chi_{12}}{\chi_{10}}
\right)^{4}+140\left(\frac{\chi_{12}}{\chi_{10}}
\right)^{2}\mathcal{E}_{4}
-\frac{200}{3}\left(\frac{\chi_{12}}{\chi_{10}}
\right)\mathcal{E}_{6}-14\, \mathcal{E}_{4}^{2} \right).
\end{align*}
Note that the prefactor is given by 
\[
F_{g}(0,0,0) =
\frac{12B_{2g-2}|B_{2g}|}{g(4g-4)(2g-2)!},
\]
the degree 0 genus $g$ Gromov-Witten invariant of $X_{\ban}$.  See
\cite{Pietromonaco_2018} for details of this computation.

\begin{remark}\label{rem: Fg are genus 2 Siegel forms have mirror
symmetry interpretation}
Theorem~\ref{thm: Fg are Siegel mod forms} has a nice interpretation
in terms of mirror symmetry. The Gromov-Witten potentials are
functions of local coordinates on the K\"ahler moduli space. Under mirror
symmetry, these become coordinates on the complex moduli space of the
mirror. Since the arguments of a genus 2 Siegel modular form are
coordinates on the moduli space of genus 2 curves (or Abelian
surfaces), we expect the complex moduli space of $\check{X}_{\ban}$,
the mirror of the banana manifold, to contain a subspace isomorphic to
the moduli space of genus 2 curves. Indeed, it has already been
observed that the mirror of a \emph{local} banana configuration should
be a genus 2 curve
\cite{Abouzaid-Auroux-Katzarkov,Gross-Katzarkov-Ruddat,Ruddat}.
\end{remark}

\begin{remark}
We also determine the genus 0 and genus 1 potentials. Up to degree 0
terms (which are unstable for $g=0$ or $g=1$), $F_{0}$ is the
Skoruppa-Maass lift of the constant $12$ (viewed as a weight 0 modular
form) and $F_{1}$ is the Maass lift of $12\wp\cdot \phi_{-2,1} $ where
$\wp$ is the Weierstrass $\wp$-function\footnote{Interestingly, the Jacobi form
$\phi_{0,1}=12\wp \cdot \phi_{-2,1} $ is also equal to $\half \Ell
(K3)$. }.
\end{remark}

\subsection{Modular forms and lifts}\label{subsec: modular forms and lifts}

\begin{definition}\label{defn: modular form}
A weight $k$ \emph{modular form} is a holomorphic function $f(\tau )$
on $\HH =\{\tau \in \CC , \operatorname{Im} \tau >0 \}$ satisfying
\begin{itemize}
\item For all $\left(\begin{smallmatrix} a&b\\
c&d \end{smallmatrix} \right)\in SL_{2}(\ZZ )$
\[
f\left(\frac{a\tau +b}{c\tau +d} \right) = (c\tau +d)^{k}f(\tau ),
\]
\item $f(\tau )$ admits a Fourier series of the form
\[
f(\tau ) = \sum_{n=0}^{\infty} a_{n}q^{n} ,\quad \quad q=e^{2\pi
i\tau}. 
\]
\end{itemize}
We denote the space of weight $k$ modular forms by $\Mod_{k}$.  We
sometimes abuse notation by writing $f(q)$ for the Fourier expansion
of $f(\tau )$.
\end{definition}

The $2g$-th Eisenstein series is given by
\[
E_{2g}(q) = 1 - \frac{4g}{B_{2g}} \sum_{n=1}^{\infty}\sum_{d|n}
d^{2g-1} q^{n}
\]
where $B_{2g}$ is the $2g$-th Bernoulli number. $E_{2g}$ is a modular
form of weight $2g$ for all $g\geq 2$. 

\begin{definition}\label{defn: Jacobi form}
A \emph{weak Jacobi form} of weight $k$ and index $m$ is a holomorphic
function $\phi (\tau ,z)$ on $\HH \times \CC$ satisfying
\begin{itemize}
\item For all $\left(\begin{smallmatrix} a&b\\
c&d \end{smallmatrix} \right)\in SL_{2}(\ZZ )$
\[
\phi \left(\frac{a\tau +b}{c\tau +d},\frac{z}{c\tau +d} \right) =
(c\tau +d)^{k} e^{\frac{2\pi icmz^{2}}{c\tau +d}}\phi (\tau,z ),
\]
\item  for all $u,v\in \ZZ$, 
\[
\phi (\tau ,z+u\tau +v) = e^{-2\pi im(\tau u^{2}+2 zu)} \phi (\tau ,z),
\]
\item and $\phi$ admits a Fourier expansion of the form
\[
\phi (\tau ,z) = \sum_{n=0}^{\infty}\sum_{l\in \ZZ} c_{\phi }(n,l)q^{n} y^{l}
\]
where $q=e^{2\pi i\tau}$ and $y=e^{2\pi iz}$. 
\end{itemize}
\end{definition}
In the case of index $m=1$, the Fourier coefficients $c_{\phi }(n,l)$
only depend on $4n-l^{2}$ and so we 
will sometimes in this case write
\[
 c_{\phi }(4n-l^{2}) = c_{\phi }(n,l) .
\]
We denote the space of weak Jacobi forms of weight $k$ and index $m$
by $\Jac_{k,m}$.  We sometimes abuse notation by writing $\phi
(q,y)$ for the Fourier expansion of $\phi (\tau ,z)$. A basic example
is given by $\phi_{-2,1}$ whose Fourier expansion is given by
\[
\phi_{-2,1}(q,y) = y^{-1}(1-y)^{2}\prod_{n=1}^{\infty}
\frac{(1-yq^{n})^{2}(1-y^{-1}q^{n})^{2}}{(1-q^{n})^{4}}.
\]
Up to a multiplicative constant, $\phi_{-2,1}$ is the unique weak
Jacobi form of weight -2 and index 1. We also will use the Weierstrass
$\wp$-function:
\[
\wp(q,y) = \frac{1}{12} +\frac{y}{(1-y)^{2}} +\sum_{n=1}^{\infty}
\sum_{d|n} d (y^{d}+y^{-d}-2)q^{n}
\]
and we note that up to a multiplicative constant
\[
\phi_{0,1}=12\cdot \phi_{-2,1}\cdot \wp
\]
is the unique weak Jacobi form of weight 0 and index 1.

The product of a weak Jacobi form of weight $k$ and index $m$ with a
modular form of weight $n$ is a weak Jacobi form of $k+n$ and index
$m$. In particular, multiplication by $\phi_{-2,1}$ defines a map
which we call the \emph{Skoruppa lift}:
\[
\begin{tikzcd}[column sep = large]
\Mod_{2g} \arrow[r,"\text{Skoruppa}"]&
\Jac_{2g-2,1}.
\end{tikzcd}
\]

\begin{definition}\label{defn: Hecke operator}
Let $k$ be even. For $m$ a non-negative integer, the \emph{$m$-th Hecke operator}
\[
V_{m}: \Jac_{k,1}\to \Jac_{k,m}
\]
is given by taking
\[
\phi (q,y) = \sum_{n=0}^{\infty} \sum_{l\in \ZZ} c_{\phi }(n,l)q^{n}y^{l} 
\]
to 
\[
(\phi |V_{m}) = \sum_{n=0}^{\infty} \sum_{r\in \ZZ} \sum_{d|(n,r,m)}
d^{k-1} c_{\phi }\left(\frac{nm}{d^{2}},\frac{r}{d} \right)q^{n}y^{r}
\]
for $m>0$ and
\[
(\phi |V_{0}) = c_{\phi }(0,0)\frac{-B_{k}}{2k} +\sum_{n=0}^{\infty} \sum_{\begin{smallmatrix}r\in \ZZ\\
r>0\text{ if } n=0  \end{smallmatrix}} \sum_{d|(n,r)}
d^{k-1} c_{\phi }\left(0,\frac{r}{d} \right)q^{n}y^{r}
\]
for $m=0$. 
\end{definition}

\begin{definition}\label{defn: Maass lift}
Let $\phi \in \Jac_{k,1}$ with $k$ even. The \emph{Maass lift} of
$\phi$ is given by
\[
ML(\phi ) = \sum_{m=0}^{\infty} (\phi |V_{m})Q^{m}. 
\]
\end{definition}

The following is due to Eichler-Zagier \cite{Eichler-Zagier} in the
case of holomorphic Jacobi forms, and Borcherds
\cite[Thm~9.3]{Borcherds-Inventiones} and Aoki \cite{Aoki} in the
case of weak Jacobi forms\footnote{We thank H. Aoki and G. Oberdieck
for discussion on this point.}:

\begin{theorem}\label{thm: Maass lift of a weak Jac form is a mero
Siegel form} The Maass lift of a weak Jacobi form of weight $k>0$ and
index 1 is a meromorphic genus 2 Siegel modular form of weight $k$. If
the Jacobi form is holomorphic, then the Maass lift is also
holomorphic. Here $Q=e^{2\pi i \sigma}$, $q=e^{2\pi i \tau }$, and
$y=e^{2\pi iz}$ where $\left(\begin{smallmatrix} \tau  &z\\
z&\sigma \end{smallmatrix} \right)\in \HH_{2}$.
\end{theorem}

We may reformulate the Maass lift in terms of polylogarithms. Let
\[
\Li_{a}(x) = \sum_{n=1}^{\infty} n^{-a}x^{n}.
\]
Then a straightforward computation yields the following
\begin{lemma}\label{lem: Maass  lift in terms of polylogs}
Let $\phi =\sum_{n=0}^{\infty} \sum_{l\in \ZZ} c_{\phi }(4n-l^{2})q^{n} y^{l} \in
\Jac_{k,1}$ with $k$ even. Then
\[
ML(\phi ) = c_{\phi }(0)\frac{-B_{k}}{2k} + \sum_{n,m\geq 0}\sum_{\begin{smallmatrix} l\in \ZZ\\
l>0\text{ if }(n,m)=(0,0) \end{smallmatrix}} c_{\phi }(4nm-l^{2})\Li_{1-k}
(Q^{m}q^{n}y^{l}). 
\]
\end{lemma}

\subsection{The $\lambda$ expansion of $\Ell_{q,y}(\CC^{2} ,t)$.}

Recall from \S~\ref{sec: elliptic genera of Hilb(C2)} that the
coefficients $c(d,k)$ are defined by the expansion of the equivariant
elliptic genus of $\CC^{2}$:
\[
\Ell_{q,y}(\CC^{2},t) = \sum_{n=0}^{\infty} \sum_{l,k\in \ZZ}
\, c(4n-l^{2},k)\, q^{n}y^{l} t^{k} .
\]
Let $t=e^{i\lambda}$. Theorem 4.4 in Zhou
\cite{Zhou-regularized-elliptic-genera} gives the expansion of
$\Ell_{q,y}(\CC^{2},t)$ as a Laurent series in $\lambda$. His result
is\footnote{ We use the $r=1$ case of Zhou's theorem. Our $\lambda$ is
$2\pi t$ in Zhou's notation (his $t$ is not our $t$). There is a typo in his formula: the $\eta (\tau )$ should be $\eta (\tau
)^{3}$. }
\[
\Ell_{q,y}(\CC^{2}, e^{i\lambda}) = \lambda^{-2} \cdot
\phi_{-2,1}(q,y) \cdot \left(1+\wp(q,y)\lambda^{2}+\sum_{g=2}^{\infty}
\frac{|B_{2g}|}{2g(2g-2)!} E_{2g}(q)\lambda^{2g}  \right).
\]

Let
\[
\psi_{2g-2}(q,y) = \operatorname{Coef}_{\lambda^{2g-2}}
\left[\Ell_{q,y}(\CC^{2},e^{i\lambda}) \right]
\]
so that Zhou's result can be expressed as
\[
\psi_{2g-2}(q,y) = \phi_{-2,1}(q,y)\cdot \begin{cases}
1&g=0\\
\wp(q,y)&g=1\\
\frac{|B_{2g}|}{2g(2g-2)!} E_{2g}(q) &g>1
\end{cases}
\]

We observe that $\psi_{2g-2}$ is a weak Jacobi form of weight $2g-2$
and index 1 and consequently has an expansion
\[
\psi_{2g-2}(q,y) = \sum_{n=0}^{\infty} \sum_{l\in \ZZ} c_{2g-2}(4n-l^{2})q^{n}y^{l},
\]
which defines the coefficients $c_{2g-2}(d)$.

Comparing coefficients in the $\lambda$ and the $t$ expansions of
$\Ell_{q,y}(\CC^{2},t=e^{i\lambda})$ we get the following fundamental
relationship between the coefficients $c_{2g-2}(d)$ and $c(d,l)$:
\begin{equation}\label{eqn: c_{2g-2}(d) compared to c(d,l)}
\sum_{g=0}^{\infty} c_{2g-2}(d)\lambda^{2g-2} = \sum_{l\in \ZZ}
c(d,l)e^{il\lambda }.
\end{equation}

\subsection{Gromov-Witten potentials via the GW/DT 
correspondence}\label{subsec: F_{g} via GW/DT}

$\quad $

The GW/DT correspondence is a conjectural equivalence between the
Gromov-Witten and the Donaldson-Thomas invariants of a Calabi-Yau
threefold \cite{MNOP1}. It has been proven for a broad class of
Calabi-Yau threefolds including complete intersections in products of
projective spaces \cite{Pandharipande-Pixton-DTGW}, which
unfortunately does not include $X_{\ban}$.

However, if we assume that the GW/DT correspondence holds for
$X_{\ban}$, we may compute the genus $g$ Gromov-Witten potentials from
our formula for the Donaldson-Thomas partition function
(Theorem~\ref{thm: formula for Zban}). We define the reduced genus $g$
Gromov-Witten potential (for banana curve classes) of $X_{\ban}$ by
\[
F'_{g}(Q_{1},Q_{2},Q_{3}) = \sum_{\dvec >0}
GW^{g}_{\dvec } (X_{\ban}) Q_{1}^{d_{1}} Q_{2}^{d_{2}}
Q_{3}^{d_{3}} .
\]
Here $GW^{g}_{\dvec }(X_{\ban})$ denotes the genus $g$
Gromov-Witten invariant in the class
$\beta_{\dvec}= d_{1}C_{1}+d_{2}C_{2}+d_{3}C_{3}$ and $\dvec >0$ means
$d_{i}\geq 0$ and  $(d_{1},d_{2},d_{3}) \neq (0,0,0)$.

The GW/DT correspondence asserts that 
\[
\sum_{g=0}^{\infty} F'_{g}(Q_{1},Q_{2},Q_{3})
\lambda^{2g-2} = \log
\left(\frac{Z_{\Gamma}(X_{\ban})}{Z_{\Gamma}(X_{\ban})|_{Q_{i}=0}}  \right) 
\]
under the change of variables $p=e^{i\lambda}$.

We now prove Theorem~\ref{thm: Fg are Siegel mod forms}. Applying
Theorem~\ref{thm: formula for Zban} and using Equation~\eqref{eqn:
c_{2g-2}(d) compared to c(d,l)}, we get
\begin{align*}
\sum_{g=0}^{\infty} F'_{g}(Q_{1},Q_{2},Q_{3}) \lambda^{2g-2} &= \log
\left(\prod_{\dvec >0}\, \prod_{k\in \ZZ} (1-p^{k}
Q_{1}^{d_{1}}Q_{2}^{d_{2}}Q_{3}^{d_{3}})^{-12c(||\dvec ||,k)} \right)\\
&=\sum_{\dvec >0} \, \sum_{k\in \ZZ} 12c(||\dvec
||,k)\sum_{n=1}^{\infty} \frac{1}{n} p^{nk} Q_{1}^{nd_{1}} Q_{2}^{nd_{2}} Q_{3}^{nd_{3}}\\
&=12 \sum_{\dvec >0} \sum_{n=1}^{\infty} \frac{1}{n} Q_{1}^{nd_{1}}
Q_{2}^{nd_{2}} Q_{3}^{nd_{3}} \sum_{g=0}^{\infty} c_{2g-2} (||\dvec
||) n^{2g-2} \lambda^{2g-2}. 
\end{align*}

Thus we find that
\begin{align*}
F'_{g}(Q_{1},Q_{2},Q_{3}) &= 12 \sum_{\dvec >0} c_{2g-2}(||\dvec ||)
\sum_{n=1}^{\infty} n^{2g-3}  Q_{1}^{nd_{1}}
Q_{2}^{nd_{2}} Q_{3}^{nd_{3}}\\
&= 12 \sum_{\dvec >0} c_{2g-2} (||\dvec ||) \Li_{3-2g} (  Q_{1}^{d_{1}}
Q_{2}^{d_{2}} Q_{3}^{d_{3}} ) .
\end{align*}

Substituting
\[
Q_{1} = e^{2\pi iz} = y,\quad \quad Q_{2} = e^{2\pi i(\tau -z)} =
qy^{-1},\quad \quad  Q_{3} = e^{2\pi i(\sigma  -z)} = Qy^{-1},
\]
reindexing by 
\[
d_{1} = l+n+m,\quad d_{2}=n,\quad d_{3}=m,
\]
so that $ ||\dvec ||=4nm-l^{2}$, and noting that $\dvec >0$ is
equivalent to $n,m\geq 0,l\in \ZZ $ and $l>0$ if $n=m=0$, we find that
\[
F_{g}'(q,y,Q) = 12 \sum_{n,m\geq 0} \sum_{\begin{smallmatrix} l\in \ZZ\\
l>0\text{ if }(n,m)=(0,0) \end{smallmatrix}} c_{2g-2}(4nm-l^{2}) \Li_{3-2g}
(Q^{m}q^{n}y^{l}). 
\] 

For $g>1$, the full genus $g$ Gromov-Witten potential is the reduced
potential plus the constant term:
\[
F_{g} = GW^{g}_{\zerovec} (X_{\ban}) + F'_{g}.
\]
Using for example the formulas in \cite[\S~2.1]{MNOP1} we know
\begin{align*}
GW^{g}_{\zerovec}(X_{\ban})& = (-1)^{g}\frac{1}{2} e(X_{\ban})
\frac{|B_{2g}|\cdot |B_{2g-2}|}{2g(2g-2)(2g-2)!} \\
&=12\cdot \left(\frac{-B_{2g-2}}{4g-4} \right) \cdot
\left(\frac{-|B_{2g}|}{g(2g-2)!} \right). 
\end{align*}
Examining the $y^{0}q^{0}$ term of $\psi_{2g-2}(q,y)=\phi_{-2,1}\cdot
\frac{|B_{2g}|}{2g(2g-2)!}E_{2g}$ for $g>1$ yields
\[
c_{2g-2}(0) = \frac{-|B_{2g}|}{g(2g-2)!}
\]
and hence we find
\[
F_{g} = 12\cdot \left(c_{2g-2}(0)\frac{-B_{2g-2}}{4g-4} +  \sum_{n,m\geq 0} \sum_{\begin{smallmatrix} l\in \ZZ\\
l>0\text{ if }(n,m)=(0,0) \end{smallmatrix}} c_{2g-2}(4nm-l^{2})
\Li_{3-2g} (Q^{m}q^{n}y^{l}) \right).
\]
By Lemma~\ref{lem: Maass  lift in terms of polylogs}, the above is
exactly the
Maass lift of $12 \psi_{2g-2}$, and hence we find that $F_{g}$ is the
Skoruppa-Maass lift of $a_{g}E_{2g}$ where
\[
a_{g} = \frac{6|B_{2g}|}{g(2g-2)!}. 
\]
This completes the proof of Theorem~\ref{thm: Fg are Siegel mod
forms}.

\begin{remark}\label{rem: F0 and F1 are formal Skoruppa-Maass lifts}
Note that the proof of the theorem also shows that up to constant
terms, the genus 0 and the genus 1 Gromov-Witten potentials are given
by the Maass lifts of $12\phi_{-2,1}$ and $12\phi_{-2,1}\wp =
\phi_{0,1}$ respectively. Although these are weak Jacobi forms of
index 1, they are not of positive weight, and hence their Maass lifts
are not guarenteed to be Siegel forms. See \cite{Pietromonaco_2018}
for a further discussion.
\end{remark}

\subsection{Gopakumar-Vafa invariants}\label{subsec: appendix tables
of GV invariants}

In this section we give tables of values of the Gopakumar-Vafa
invariants $n^{g}_{a}(X_{\ban})$ for small values of $g$ and
$a$. Since all values are divisible by 12, we list $\frac{1}{12}\,
n^{g}_{a}(X_{\ban})$ (which can also be regarded as the Gopakumar-Vafa
invariants of a local banana configuration). We note that the non-zero
values have $a$ congruent to $0$ or $-1$ modulo 4 and so we organize the
tables as such.

\begin{center}
\begin{tabular}{|c|ccccccc|}
\hline

&&&&&&&\\

    $\frac{1}{12}\, n^{g}_{ 4n-1}(X_{\ban })$
 & {$g=0$} & {$g=1$} & {$g=2$} & {$g=3$} & {$g=4$} & {$g=5$} & {$g=6$} \\
 &&&&&&&\\
\hline 
    $n=0$  & 1     &    0     &  0    &    0     &    0     &    0   & 0   \\
    $n=1$  &  8    &   -6     &  1    &    0     &    0     &    0   &  0 \\
    $n=2$  & 39 &   -46    &  17   &    -2     &    0     &     0    &   0 \\
    $n=3$  & 152      &    -242    &   139    &    -34     &     3    &     0     &    0 \\
    $n=4$  & 513  &    -1024  & 800  &      -304   &    56     &      -4    &  0 \\ 
    $n=5$  & 1560  & -3730 & 3683  &    -1912 &    548      &      -82    &  5 \\ \hline 
\end{tabular}
\end{center}

\begin{center}
\begin{tabular}{|c|ccccccc|}
\hline

&&&&&&&\\

    $\frac{1}{12}\, n^{g}_{4n}(X_{\ban})$ & {$g=0$} & {$g=1$} & {$g=2$} & {$g=3$} & {$g=4$} & {$g=5$} & {$g=6$} \\&&&&&&&\\ \hline 
    $n=0$  & -2    &    1     &  0    &    0     &    0     &    0   &  0 \\
    $n=1$  &  -12    &    10     &  -2    &    0     &     0    &     0    &  0    \\ 
    $n=2$  &  -56  &   72  &  -30  &   4    &      0     &     0    &   0     \\
    $n=3$  & -208      &   352    &    -220  &      60     &     -6    &     0      &   0 \\ 
    $n=4$  & -684  & 1434 & -1194  &   492  &    -100      &      8    &  0 \\ 
    $n=5$  & -2032  & 5056 & -5252  &    2908 &    -902      &      148    &  -10 \\ \hline 
\end{tabular}
\end{center}

\bibliography{mainbiblio} \bibliographystyle{plain}

\end{document}